\documentclass[10pt,twocolumn]{article}

\usepackage[a4paper,margin=1.8cm,columnsep=0.6cm]{geometry}

\usepackage{graphicx}
\usepackage[dvipsnames]{xcolor}
\usepackage{amsmath,amsfonts,amssymb,amsthm}
\usepackage{tikz}
\usepackage{float}
\usepackage{cite}

\graphicspath{{images/}}

\numberwithin{table}{section}

\newtheorem{definition}{Definition}[section]
\newtheorem{theorem}{Theorem}[section]
\newtheorem{proposition}{Proposition}[section]

\newtheorem{lemma}{Lemma}[section]

\title{On Two-Player Scalar Discounted Linear Quadratic Games}

\author{
Chiara Cavalagli\\[0.5em]
\small IMT School for Advanced Studies Lucca
\and
Alberto Bemporad\\[0.5em]
\small IMT School for Advanced Studies Lucca
\and
Mario Zanon\\[0.5em]
\small IMT School for Advanced Studies Lucca
}

\date{}

\begin{document}

\maketitle

\begin{abstract}
This paper studies feedback Nash equilibria (FNE) in scalar two-player
discounted linear-quadratic games. By explicitly discounting the objective function,
we reveal the existence of finite-cost equilibria that do
not necessarily stabilize the system, motivating a distinction between
stable FNE and FNE. We provide a sufficient condition ensuring that
every FNE is stabilizing, together with an analytical characterization
of the existence, uniqueness, and multiplicity of FNE. For the
symmetric setting, closed-form expressions are derived together with a
complete characterization of multiplicity in terms of the game
parameters. The local attractive and repelling properties of a FNE under
the best-response dynamics are characterized and related to
multiplicity in the symmetric setting. Analytical and numerical results
further reveal an anti-coordination structure in the value functions
when multiple FNE exist. Finally, the Community Cost is investigated,
showing how the ordering of FNE depends on the relative weighting of
the players' objectives.
\end{abstract}

\section{Introduction}
Non-cooperative games~\cite{noncooperative} provide a framework for analyzing strategic interactions among self-interested players, with applications in economics, robotics, and communication systems. We focus on discrete-time infinite-horizon (IH) Linear--Quadratic (LQ) games, extending the classical Linear Quadratic Regulator problem~\cite{Bellman1957} to the multi-player setting. A well-known solution concept is the feedback Nash Equilibrium (FNE), whose computation typically requires solving coupled Riccati equations, for which closed-form solutions are rarely available except under restrictive structural assumptions such as zero-sum games~\cite{noncooperative}. Several iterative methods have therefore been developed to compute or approximate FNE, including Policy Iteration~\cite{policy_iteration}, Policy Gradient methods~\cite{policy_grad,mazumdar2019}, Reinforcement Learning approaches~\cite{offpolicy_q}, Riccati and Lyapunov iterations~\cite{Nortmann_iterative_methods}, and Finite-Horizon (FH) approximations~\cite{finite_infinite}. The latter provide a natural numerical approximation of the IH problem; however, the relationship between FH and IH equilibria remains an open question\cite{finite_infinite}, particularly when multiple FNE coexist. A systematic characterization of the IH solution landscape can therefore provide a basis for understanding how FH solutions relate to their IH counterparts, as recently investigated in~\cite{salizzonihall2025}. At the same time,
the scalar two-player IH game already exhibits nontrivial equilibrium structures, allowing for a detailed mathematical analysis of the game. Recent works have addressed existence, uniqueness, and multiplicity conditions for the two-player game, either numerically via Gröbner bases~\cite{Salizzoni_2025_scalar} or geometrically through the analysis of Riccati equations~\cite{Nortmann_2023_two_players}. These results not only provide insights into how the equilibria depend on the game parameters and may provide insights relevant to the $N$-player setting~\cite{feedbackNE}, but are also crucial for understanding the behavior of iterative methods, which still relies largely on numerical findings.

Building on these foundations, this paper advances the theoretical understanding of scalar two-player LQ games by introducing a discount factor $\gamma\in(0,1]$ in the objective function. While discounting is standard in economic models~\cite{engwerda2005lq} and reinforcement learning, its role in the analysis of scalar LQ games has only recently been considered~\cite{cavalagli_FNE}. In Section~\ref{sec:mathematical_analysis}, we show that discounting allows for finite-cost equilibria that do not necessarily stabilize the system, yielding a distinction between stable FNE and FNE. A sufficient condition ensuring that every existing solution is also stabilizing is provided, together with a characterization of existence, uniqueness, and multiplicity of FNE, recovering past results~\cite{feedbackNE, Nortmann_2023_two_players, Salizzoni_2025_scalar} for the special case $\gamma=1$, thereby generalizing the theory available so far. In Section~\ref{sec:BR}, we introduce the notions of local attractiveness and repulsiveness of FNE with respect to the best-response dynamics and provide a complete analytical characterization of these properties. In the final Section~\ref{sec:value}, we provide an analytical and numerical characterization of the value function across multiple FNE, together with a qualitative analysis based on the Community Cost, highlighting how its ordering across FNE depends on the relative weighting of the players' objectives. In both Section~\ref{sec:mathematical_analysis} and Section~\ref{sec:BR}, we give particular emphasis to the symmetric setting, where players share identical cost parameters, and derive closed-form expressions for equilibria together with an explicit relation between local attractiveness and multiplicity.

\section*{Notation}
We denote by $\mathbb{R}$, $\mathbb{R}_{\neq 0}$, $\mathbb{R}_{\geq 0}$, and
$\mathbb{R}_{>0}$, respectively, the sets of real numbers, nonzero real numbers, nonnegative real numbers, and positive real numbers. For each player index $i$, the notation $-i$ denotes the remaining player. For notational convenience, throughout the remainder of the paper, we define $\Delta(k_i, \sigma_i)=\gamma \left(\sigma_i+k_i^2\right)^2+4k_i^2$ and $\nabla(k_i, \sigma_i, \sigma_{-i})=(\sigma_i+k_i^2)^2-4\sigma_{-i}k_i^2$ for every $k_i\neq 0,$ and for every $\sigma_i, \sigma_{-i} \in\mathbb{R}_{>0}$.

\section{Problem Statement}
Let us consider a scalar non-cooperative discrete-time system with state $x_t\in\mathbb{R}$, and control inputs $u_{i, t}\in\mathbb{R},i=1,2$. The dynamics are described by
\begin{equation}\label{eq:LQ}
        x_{t+1}  = a  x_t + \displaystyle \sum_{i=1}^{2} u_{i,t}\ ,
\end{equation}
where $a\in\mathbb{R}$ is a system parameter. Each player has a local objective function defined as
\begin{equation}\label{eq:cost_j}
J_i(x_t, u_{i,t})  = \displaystyle \displaystyle \sum_{k=t}^{\infty} \gamma^k r_i\left[ \sigma_i x_k^2 + u_{i,k}^2 \right],
\end{equation}
with $\sigma_i\in\mathbb{R}_{> 0}, r_i\in\mathbb{R}_{> 0}$ being the cost weights and $\gamma \in(0,1]$ discount factor. While discounting is a standard approach in
reinforcement learning and economic models, as explained in
\cite[Sec.~$3.6$]{engwerda2005lq} for the continuous-time counterpart,
its role in scalar linear-quadratic (LQ) games has only recently been considered in
\cite{cavalagli_FNE}. It is important to note that while the framework~\eqref{eq:cost_j} can be rewritten in the form considered in~\cite{feedbackNE, Salizzoni_2025_scalar, Nortmann_2023_two_players} via the change of variables $\tilde{a}=a\sqrt{\gamma}$, $\tilde{\sigma_i}=\gamma \sigma_i$ and $\tilde{k_i}=k_i\sqrt{\gamma}$, the discounted setting, as it will be shortly explained, generalizes the latter by recovering the undiscounted setting as a special case for $\gamma=1$. To this matter, from now on we will refer to the class of problems~\eqref{eq:LQ}-\eqref{eq:cost_j} as discounted LQ games.

In this work we restrict our attention to a specific class of strategies, namely the feedback policies, where at every instant $t$ every player applies a linear feedback law to the observable state $x_t$, i.e. $u_{i,t} = k_i x_t$, $k_i\in\mathbb{R}$. The corresponding solution concept takes the name of feedback Nash Equilibrium (FNE). By keeping a vector of policies fixed $k=(k_1, k_2)$, we can rewrite the system dynamics through the closed loop term $a_\mathrm{cl}(k)=a+k_1+k_2$ and obtain the relation $x_{t+1}=a_\mathrm{cl}^t(k) x_0$, defined for every initial state $x_0\in\mathbb{R}$. By substituting the latter inside~\eqref{eq:cost_j}, we obtain the so-called Value function of the $i$-th player: 
\begin{equation}\label{eq:value_function_0}
    V_i(x_0; k_i,k_{-i})=  x_0^2r_i\left[ \sigma_i + k_i^2  \right]\displaystyle \sum_{t=0}^{\infty}\left( \gamma a_{\mathrm{cl}}^2(k)\right)^t ,
\end{equation}
which is a geometric series of argument $y=\gamma a_\mathrm{cl}^2(k)$, being finite if and only if $|y|<1\ \iff\ |a_\mathrm{cl}(k)|<1/\sqrt{\gamma}$, possibly including non-stabilizing policies when $\gamma<1$. For example, the game $(a, \sigma_1, \sigma_2, \gamma)=(2, 0.01, 0.1, 0.6)$ admits three FNE, $(-0.6232, -0.3663)$, $(-1.0136, -0.1093)$, and $(-0.0079, -1.1893)$, and not all of them satisfy the stability condition. Thus, it is necessary to distinguish between a FNE and a stable FNE, a distinction that does not arise in the undiscounted case $\gamma=1$~\cite{feedbackNE,Salizzoni_2025_scalar}.

\begin{definition}[Stable feedback Nash Equilibrium] \label{def:stable_fne}
    Let us consider a discounted LQ game and a vector $k^\star=(k^\star_1,k^\star_2)$ such that $|a_\mathrm{cl}(k^\star)|<1/\sqrt{\gamma}$. Then we say that $k^\star$ is a FNE of the system if, for all $x_0$ and for each player we have $V_i(x_0;k^\star_i,k^\star_{-i})\leq V_i(x_0;k_i,k^\star_{-i})\ \forall k_i$. Moreover, we state it is a stable FNE if, in addition, we have $|a_{\mathrm{cl}}(k^\star)|<1$. 
\end{definition}

We remark that, since Definition~\ref{def:stable_fne} holds for every initial state $x_0$, the resulting FNE is strongly time-consistent~\cite[Def.~$5.14$]{noncooperative}. Moreover, unless additional conditions are considered, the discounted LQ game has no guarantees that all existing FNE are also stabilizing.  For the case of $\gamma=1$, the two solution concepts coincide as $|a_\mathrm{cl}(k^\star)|<1/\sqrt{\gamma}\iff |a_\mathrm{cl}(k^\star)|<1$, recovering the definitions of the undiscounted setting~\cite{Nortmann_2023_two_players, Salizzoni_2025_scalar}.

Next, the following Proposition provides a formal characterization of every stable FNE, as the authors did in~\cite{Salizzoni_2025_scalar, Nortmann_2023_two_players} for the undiscounted setting, and as in~\cite[Prop. 1]{cavalagli_FNE} for discounted N-player game.

\begin{proposition}[Characterization of Stable FNE]
Let us consider a discounted LQ game and a vector $k^\star$ such that $|a_\mathrm{cl}(k^\star)|<1/\sqrt{\gamma}$. Then we say that $k^\star$ is a stable FNE of the game if and only if the following system is satisfied:
\begin{equation}\label{eq:best_response}
\begin{cases}
    f_i(k_i, k_{-i})=0\quad i=1,2\\
    |a_\mathrm{cl}(k)| < 1 
\end{cases},
\end{equation}
with $f_i (k_i, k_{-i})=-\gamma k_i^{2} (a  + k_{-i}) + \gamma\sigma_i ( a + k_{-i} ) - k_i ( \gamma (a+k_{-i})^2 - \gamma \sigma_{i} - 1)$ being the first order condition (FOC) characterizing the best-response of the $i$-th player.     
\end{proposition}
\begin{proof}
    Let us consider the vector $k=(k_1, k_2)=(k_i, k_{-i})$ and the Value function of the $i$-th player defined in~\eqref{eq:value_function_0}. Thanks to our assumption, the geometric series is finite, yielding the following closed form expression
    \begin{equation}\label{eq:value_function}
    V_i(x_0, k)=\frac{x_0^2r_i\left( \sigma_i + k_i^2  \right)}{1-\gamma a_{\mathrm{cl}}^2(k)},\quad i=1,2.
    \end{equation}
    The policy $k_i$ is a stationary point of the $i$-th function if 
    \begin{equation}\label{eq:nec}
        \frac{\partial V_i }{\partial k_i}=\frac{2x_0^2 r_if_i(k_i, k_{-i})}{\left(1-\gamma a_\mathrm{cl}^2(k)\right)^2}=0
    \end{equation}
    whose full computation can be found in \cite[Lemma $1$]{Salizzoni_2025_scalar} for the undiscounted version, and extends directly to the discounted setting. In particular, by omitting the denominator of  \eqref{eq:nec} as nonzero by assumption, and by considering and re-arranging the numerator, we obtain the coupled equation in \eqref{eq:best_response}. By computing the second-order condition, that is $\frac{\partial^2 V_i}{\partial k_i^2}$, we obtain $\frac{\partial^2 V_i}{\partial k_i^2}=2 l_i(k_i, \sigma_i)/{\left(1-\gamma a^2_\mathrm{cl}(k)\right)^3}$ whose sign depends solely on
    \begin{equation*}
        \begin{aligned}
            l_i(k_i, \sigma_i)= &\gamma k_i^2 \left(1+ 3 \gamma a_\mathrm{cl}^2(k)\right) + \gamma 4k_i a_\mathrm{cl}(k) \left(1 - \gamma a_\mathrm{cl}^2(k) \right) \\
            &+ \gamma\sigma_i\left(1+3\gamma a^2_\mathrm{cl}(k)\right)+\left(1-\gamma a^2_\mathrm{cl}(k)\right)^2.
        \end{aligned}
    \end{equation*}
    For a fixed $a_\mathrm{cl}(k)$, $l_i(k_i,\sigma_i)$ is a quadratic in $k_i$ of positive leading term and discriminant $ -4\gamma[\gamma\sigma_i\left(1+3\gamma a^2_\mathrm{cl}(k)\right)^2$     
    $+\left(1-\gamma a^2_\mathrm{cl}(k)\right)^3]<0$ whenever $|a_\mathrm{cl}(k)|<1/\sqrt{\gamma}$. Hence, $l_i(k_i,\sigma_i)>0$ and therefore every stationary point is a strict local minimum, so that the condition $f_i(k_i,k_{-i})=0$ characterizes the best response of the $i$-th player.
\end{proof}

Numerical methods are available to find the roots of system \eqref{eq:best_response}, such as those proposed in~\cite{Salizzoni_2025_scalar} for the undiscounted scalar game and in~\cite{cavalagli_FNE} for the discounted scalar game. In contrast, this work focuses on the theoretical analysis of the game. Following the same structure as~\cite{ENGWERDA2016364} for the continuous-time
counterpart, and later extended in~\cite{Nortmann_2023_two_players} to the two-player case, this paper studies necessary and sufficient conditions for the existence, uniqueness, and multiplicity of FNE.

\section{Characterization of Feedback Nash Equilibria}\label{sec:mathematical_analysis}
\subsection{General Setting}
We first show that the players admit the null vector as a stable FNE only when the system dynamics are null, as discussed in~\cite[Remark $4$]{Nortmann_2023_two_players}. Thanks to that, we exclude this trivial case in the remainder of the paper by assuming $a\neq 0$, such that $k_i\neq0,\ \forall i$. 
\begin{lemma}\label{lemma:exclude_null}
    Let $(k_i^\star, k_{-i}^\star)$ be a stable FNE. Then
    \begin{equation*}
       \forall\ i:\ k_i^\star=0\ \iff\ a= 0
    \end{equation*}
\end{lemma}
\begin{proof}
    The proof is given in Appendix~\ref{appendix:exclude_null}. 
\end{proof}
In the next Lemma, we formalize the structure of the coupled equations in \eqref{eq:best_response}, by showing a factorization that will allow us to study the solutions more easily.
\begin{lemma}\label{lemma:factorization_sys}
    Let $f_i(k_i, k_{-i})$ be the FOC of the $i$-th agent. Then
\begin{equation}\label{eq:f_factorization}
    f_i(k_i, k_{-i})=-\gamma k_i(k_{-i}-h_i^-(k_i))(k_{-i}-h_i^+(k_i)),    
\end{equation}
where 
\begin{equation}\label{eq:roots}
\begin{aligned}
        \displaystyle h_i^\pm(k_i)=&\frac{\sigma_i-k_i(2a+k_i)}{2k_i} \pm\frac{\sqrt{\gamma \Delta(k_i, \sigma_i)}}{2\gamma k_i},
\end{aligned}
\end{equation}
are the two roots of $f_i(k_i, k_{-i})$ solved w.r.t. $k_{-i}$. 
\end{lemma}

Given that $f_i$ can always be factorized as in~\eqref{eq:f_factorization}, we focus next on the relation between each root and the stability condition $|a_\mathrm{cl}(k)|<1$, as discussed in \cite{Salizzoni_2025_scalar, Nortmann_2023_two_players} for the undiscounted case, where we exclude one of the two roots. In addition, we provide a sufficient condition ensuring that every FNE is also stabilizing, recovering the results of \cite[Prop. $2$]{cavalagli_FNE} for the two-player game.

\begin{proposition}\label{prop:roots_factorization}
Let $k^\star=(k_i^\star, k_{-i}^\star)$ be a FNE of a discounted LQ game. 
\begin{equation}\label{eq:NEC_roots}
    k^\star\ \text{is a stable FNE}\Rightarrow k^\star_{-i}=h_i^-(k_i^\star)\ \forall i,
\end{equation}
    \begin{equation}\label{eq:SUFF_roots}
        k^\star\ \text{is a stable FNE}\Leftarrow \exists\ i:\ \begin{cases}
            &k_{-i}^\star=h_i^-(k_i^\star)\\
            &\sigma_i>\frac{(1-\gamma)^2}{4\gamma^2}
        \end{cases}.
    \end{equation}
\end{proposition}
\begin{proof}
The result follows directly from~\cite[Proposition 2]{cavalagli_FNE}
by specializing the $N$-player formulation to the two-player case and noting that $S_{-i}\equiv k_{-i}$.

\end{proof}

\begin{figure*}[t!]
    \begin{center}    
    \includegraphics[width=\linewidth]{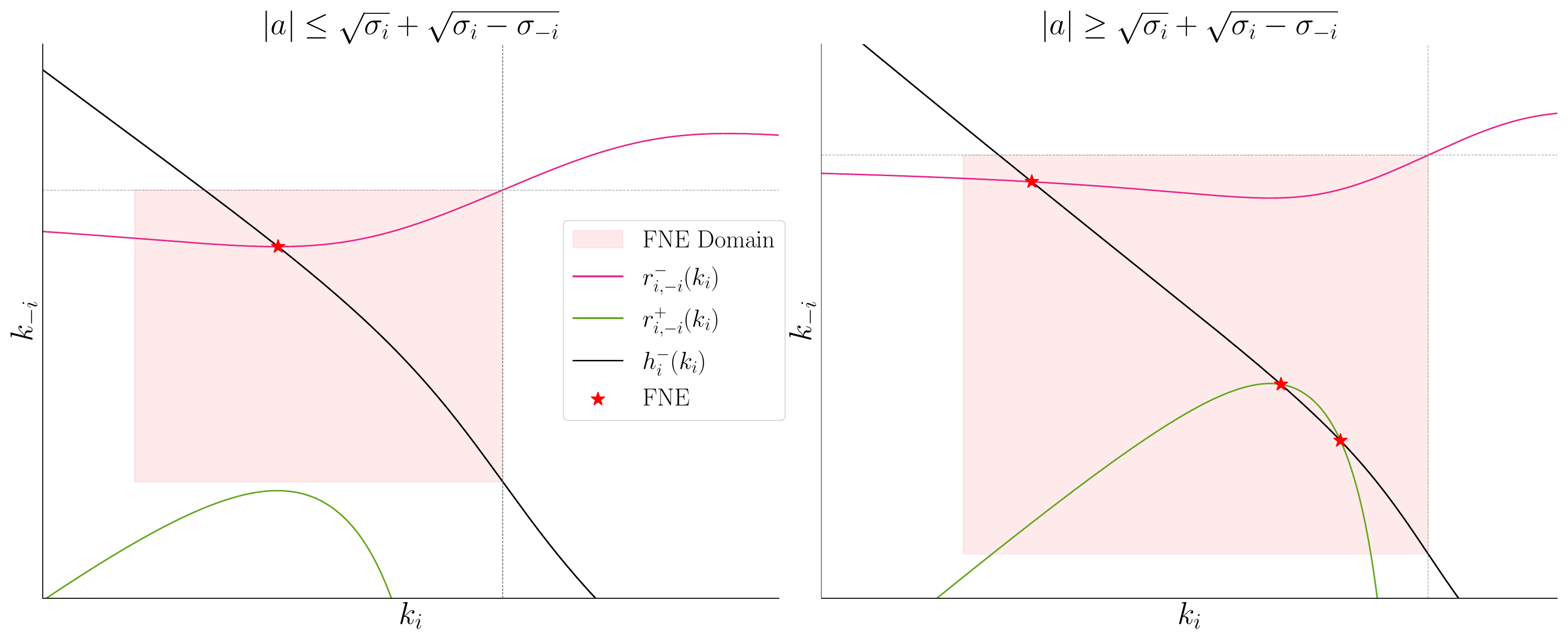}
    \caption{Two cases in which we have respectively unique and multiple FNE given $\sigma_i>\sigma_{-i}$. Any solution is the intersection between the optimal reduced curve $k_{-i}=h_i^-(k_i)$~\eqref{eq:optimal_system} and the branches $k_{-i}=r_{i, -i}^\pm(k_i)$~\eqref{eq:FNE_law_factorized}. The pink-colored area indicates the domain of existence of the FNE, formally the open square $\left(-a, 0\right)\times \left(-a, 0\right)\subset\mathbb{R}^2$. }
    \label{fig:suff_cond}
    \end{center}
\end{figure*}

We remark that for the case of $\gamma=1$, the sufficient~\eqref{eq:SUFF_roots} condition coincides with the necessary one~\eqref{eq:NEC_roots}, recovering the result of~\cite[Lemma $2$]{Salizzoni_2025_scalar}. Moreover, if the sufficient condition holds, then the distinction between stable FNE and FNE is no longer needed. Numerically, the condition becomes the trivial inequality $\sigma_i>0$ as $\gamma \rightarrow1$, which is the case of most applications, for example, for $\gamma=0.9$ it
would require $\sigma_i> 3\cdot 10^{-3}$. Thus, from now on, we will assume the sufficient condition to hold and refer to every stable FNE as FNE, so to remain consistent with the LQ theory.

Now, the original system has been reduced to its optimal form:
\begin{equation}\label{eq:optimal_system}
        \forall i:\quad k_{-i}^\star =h_i^-(k_i^\star),
\end{equation}
from which one can define the domain of existence of every FNE, recovering the result of \cite[Lemma $2$]{Salizzoni_2025_scalar} for the undiscounted game.
\begin{proposition}\label{prop:existence}
    Let $(k_i^\star, k_{-i}^\star)$ be a FNE. Then
    \begin{equation}\label{eq:sFNE_dom}
        k_i^\star\in\left(\min\{-a,0\}, \max\{-a, 0\}\right), \ \forall i.
    \end{equation}
\end{proposition}
\begin{proof}
To prove the claim, we first establish the monotonicity of the reduced
optimal roots and then compute their intersection points in the plane
$(k_i,k_{-i})$. Let us consider the player index $-i$ and $k_{-i}=h_i^-(k_i)$. Let us compute the first derivative $h_i^{-\prime}(k_i)= - \frac{\left(\sigma_{i} + k_i^{2}\right) L_i(k_i)}{2 k_i^{2} \left(\Delta(k_i, \sigma_i)\right)}$, with $L_i(k_i)=\Delta(k_i, \sigma_i)+(k_i^2- \sigma_{i} )\sqrt{\gamma \Delta(k_i, \sigma_i)}$. By factoring out $\sqrt{\gamma\Delta(k_i, \sigma_i)}$ one obtains $L_i(k_i)=\sqrt{\gamma\Delta(k_i, \sigma_i)}\left[\sqrt{\Delta(k_i, \sigma_i)/ \gamma}+k_i^2-\sigma_i\right]$. By observing that $\Delta(k_i, \sigma_i)/ \gamma\geq (\sigma_i+k_i^2)^2$ then it implies that $L_i(k_i)\geq \sqrt{\gamma\Delta(k_i, \sigma_i)}\left[2k_i^2\right]\geq 0$. Hence, 
$h_i^{-\prime}(k_i)<0$, and therefore $h_i^-(k_i)$ is strictly decreasing on its domain. In order to find the intersection points in the plane $(k_i, k_{-i})$, let us first observe that $h_i^-(0)=-a$, as fully derived in \ref{appendix:existence}; while for the other intersection, one can show that $h_i^-(k_i)=0$ yields as unique admissible point $\overline{k_i}_-(\sigma_i)= \frac{\gamma a^2+\gamma \sigma_i+1-2\gamma a^2 - \sqrt{\gamma^2(a^2+\sigma_i)^2 - 2 \gamma a^{2}  + 2 \gamma \sigma_{i} + 1}}{2 \gamma a }$. By iterating the same procedure for the other player index $i$, we obtain a mirrored behavior summarized as follows:
\begin{equation}\label{eq:sys_domain}
\begin{array}{cc}
\begin{cases}
h_i^-(0)=-a  \\
h_i^-(\overline{k}_-(\sigma_i))=0 
\end{cases}
&
\begin{cases}
h_{-i}^-(0)=\overline{k}_-(\sigma_{-i})\\
h_{-i}^-(-a)=0
\end{cases}
\end{array},
\end{equation}
and displayed in Figure~\ref{fig:domain}. As a result, the two optimal roots $h_i^-(k_i)$, $h_{-i}^-(k_{-i})$ can only intersect each other in the open square $\left(\mathrm{max}\{-a, \overline{k}_-(\sigma_i), \overline{k}_-(\sigma_{-i})\}, 0\right)^2$. Moreover, because $\overline{k}_-(\sigma_i)>-a$ for every $\sigma_i$, then the claim follows.
Full derivations can be found in the Appendix \ref{appendix:existence}. 
\end{proof}
\begin{center}
\begin{figure}[t]
    \centering
    \includegraphics[scale=0.5]{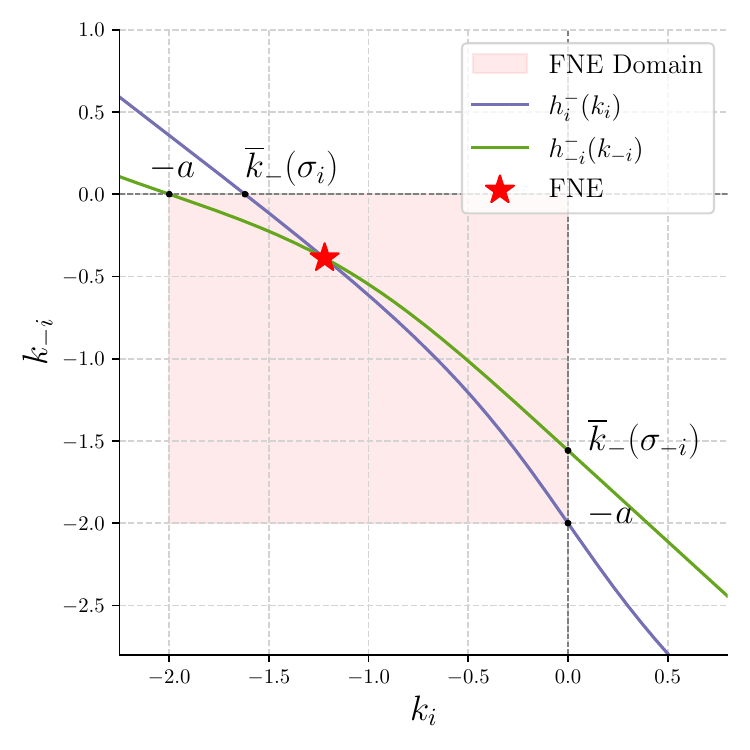}
    \caption{Illustration of the reduced optimal roots $h_i^-(k_i)$ and $h_{-i}^-(k_{-i})$, together with the boundary points of system~\eqref{eq:sys_domain} used to determine the FNE domain, highlighted in pink. The red star denotes a FNE.}
    \label{fig:domain}
\end{figure}
\end{center}
We continue the characterization of FNE by taking the factorized optimal system~\eqref{eq:optimal_system}, which can be solved analytically by removing the square roots in~\eqref{eq:roots}. After a few algebraic manipulations, we obtain an algebraic condition characterizing the FNE law:
\begin{equation}\label{eq:FNE_law}
    p_{i, -i}(k_i, k_{-i}):=\frac{\sigma_i}{k_i}+k_i-\left(\frac{\sigma_{-i}}{k_{-i}}+k_{-i}\right)=0,
\end{equation}
by recalling that every feedback policy is nonzero by assumption - Lemma~\ref{lemma:exclude_null}. The next result shows the geometrical properties of the feedback laws. 
\begin{lemma}[Geometric structure of the FNE law]\label{lemma:facto_FNE_law}
    Let us consider the FNE curve in the plane $(k_i, k_{-i})$ defined by~\eqref{eq:FNE_law}. Then for every ordering $\sigma_i>\sigma_{-i}$ we have
    \begin{equation}\label{eq:FNE_law_factorized}
    p_{i, -i}(k_i, k_{-i})=(k_{-i}-r_{i,-i}^-(k_i))(k_{-i}-r_{i,-i}^+(k_i))=0,
\end{equation}
with $r_{i,-i}^-(k_i)$ and $r_{i,-i}^+(k_i)$ being respectively bounded and unbounded branches on their admissible domain. 
\end{lemma}
\begin{proof}
    The explicit factorization and proof are reported in Appendix~\eqref{appendix:facto_FNE_law}.
\end{proof}
Lemma~\ref{lemma:facto_FNE_law} shows that the FNE law defines two distinct curves across the $(k_i, k_{-i})$ plane: a bounded branch and an unbounded one. As shown in Figure~\ref{fig:suff_cond}, only the bounded branch always intersects the FNE domain~\eqref{eq:sFNE_dom}, while whether the unbounded branch intersects it depends on the game parameters. This geometric separation is the key to prove the next result. 

\begin{theorem}[Uniqueness and Multiplicity]\label{thm:nec_multi}
Let us consider a discounted LQ game such that condition~\eqref{eq:SUFF_roots} holds. Then, for $\sigma_i>\sigma_{-i}$
     \begin{itemize}
    \item  the game always admits a unique FNE $(k_i^\star, k_{-i}^\star)$ satisfying the branch $r_{i,-i}^-$;
    \item the game may admit multiple FNE only if
    \begin{equation}\label{eq:NEC_MULT}
        |a|\geq \sqrt{\sigma_{i}}+ \sqrt{\sigma_i-\sigma_{-i}}.
    \end{equation}
    \end{itemize}
\end{theorem}
\begin{proof}
We exploit the reduced optimal system~\eqref{eq:optimal_system} and intersect it with the bounded branch of the FNE law, $r_{i,-i}^-(k_i)$, obtaining a one-variable equation whose roots are FNE. Since the resulting function has opposite signs on the FNE domain~\eqref{eq:sFNE_dom}, the Intermediate Value Theorem guarantees existence, while strict monotonicity ensures uniqueness. Hence, multiplicity can only arise from intersections with the remaining branch $r_{i,-i}^+(k_i)$. A necessary condition for this is that its maximum/minimum lies inside the open square defined by~\eqref{eq:sFNE_dom}, as expressed by the inequality~\eqref{eq:NEC_MULT} and illustrated in Figure~\ref{fig:suff_cond}. The full proof is given in Appendix~\eqref{appendix:suff_multi}. 
\end{proof}
Theorem~\ref{thm:nec_multi} establishes that the game always admits at
least one FNE, recovering for $\gamma=1$ the existence results of~\cite[Statm.~$1$, Thm.~5]{Salizzoni_2025_scalar} and~\cite[Item~$ii$, Thm.~2]{feedbackNE}. Moreover, when $\sigma_i>\sigma_{-i}$, the FNE satisfying the law $r_{i,-i}^-$ is unique. The inequality~\eqref{eq:NEC_MULT} also provides a necessary algebraic condition for multiple FNE, complementing the sufficient conditions of~\cite[Statm.~$3,4$, Thm.~5]{Salizzoni_2025_scalar} and the geometric
characterization of~\cite[Statm.~$1$, Thm.~2]{Nortmann_2023_two_players}. A complete result would further require analyzing the intersections between $h_i^-(k_i)$ and $r_{i,-i}^+(k_i)$, which do not
appear to admit a direct closed-form algebraic characterization and are therefore naturally studied geometrically, as in~\cite{Nortmann_2023_two_players}.

Next, we focus on the symmetric setting, which was excluded from Theorem \ref{thm:nec_multi} and for which a stronger characterization can be obtained.

\subsection{Symmetric Setting}
Let us consider the special case $\sigma_i=\sigma_{-i}=\sigma\in\mathbb{R}_{>0}$, which we refer to as the symmetric setting. Similarly to the analysis developed for the undiscounted continuous-time counterpart in \cite[Sec.~$4$]{ENGWERDA2016364}, and its discrete-time $N$-player extension in \cite[Sec.~IV]{cavalagli_FNE}, this Section provides a complete characterization of the existing FNE. 
\begin{center}
\begin{figure*}[t!]
    \includegraphics[width=\textwidth]{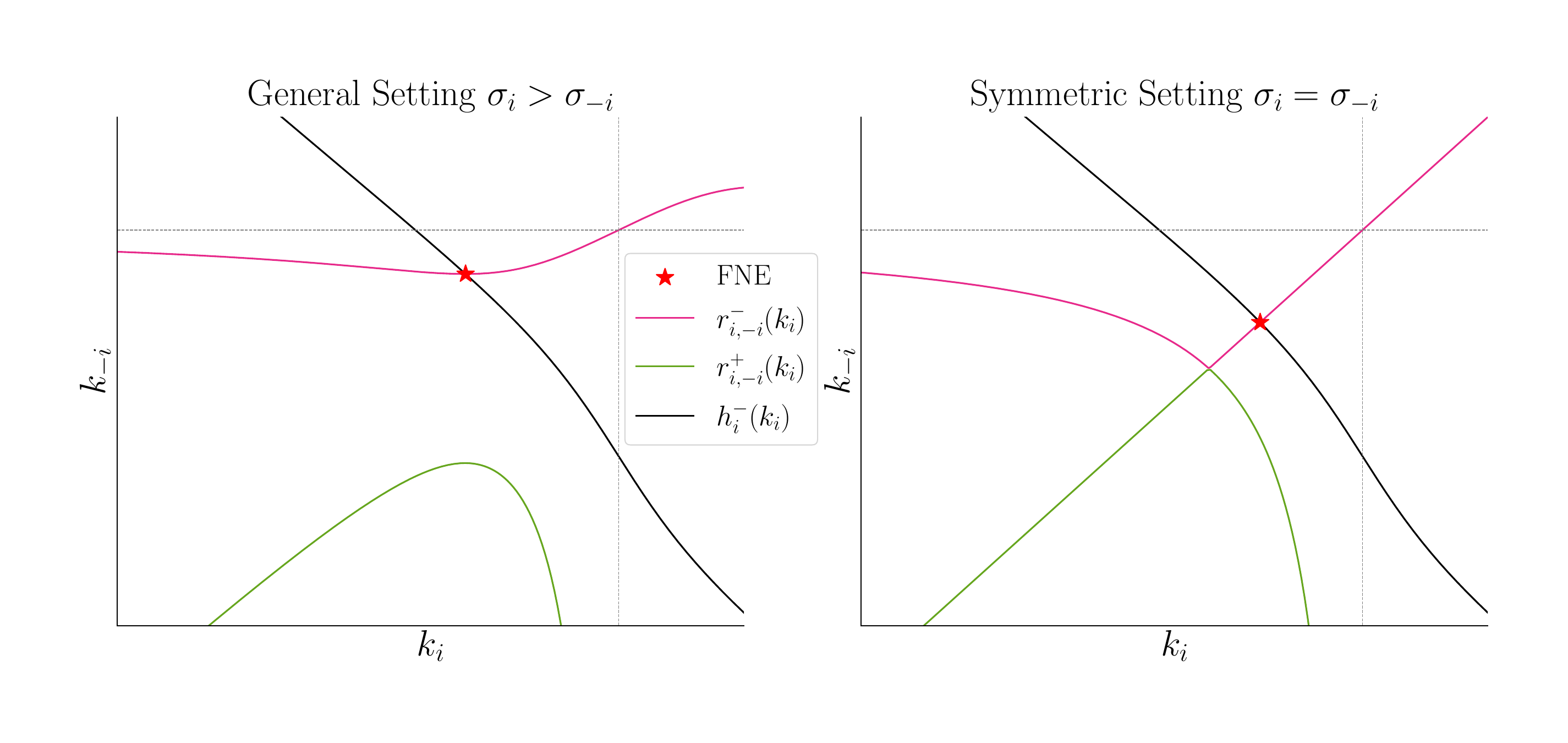}
    \caption{Factorization of the FNE law in standard and symmetric setting. In the second case the two curves $r_{i, -i}^\pm(k_i)$ contains both symmetric and hyperbolic components. }
    \label{fig:factorization}
\end{figure*}
\end{center}
Considering that the theory developed so far is still valid, in particular the sufficient condition \eqref{eq:SUFF_roots} and system~\eqref{eq:optimal_system}, for $\sigma_i=\sigma_{-i}$, the expression of~\eqref{eq:FNE_law} simplifies to $p(k_i, k_{-i})= (k_{-i}-k_i)(\sigma-k_ik_{-i})=0$. The equation states that the game admits two kinds of well-known structured FNE: a symmetric one $k_i=k_{-i}$, and/or a hyperbolic one $k_i=\frac{\sigma}{k_{-i}}$, which largely simplifies the study of existence and multiplicity. Let us take both expressions, substitute them inside the reduced optimal system~\eqref{eq:optimal_system}, and let us consider solely the numerator as we are interested into the roots of the expression; through simple algebraic manipulations we get two optimal branches, that we call respectively {\bf{symmetric branch}}
\begin{equation}\label{eq:symm_branch}
 - \sigma \gamma + \gamma k_i (2 a + 3 k_i) + \sqrt{ \gamma \Delta(k_i, \sigma)}=0,
\end{equation}
and {\bf{hyperbolic branch}}
\begin{equation}\label{eq:hyperb_branch}
\sigma \gamma + \gamma k_i \left(2 a + k_i\right) + \sqrt{\gamma \Delta(k_i, \sigma)}=0.
\end{equation}
Each FNE must satisfy one of the two expressions. The following theorem
extends the results for the general game by providing an explicit
characterization of the existence and uniqueness of the symmetric FNE,
which was established for the continuous-time counterpart in
\cite{ENGWERDA2016364}. In addition, we provide a closed-form expression
for the unique symmetric FNE, recovering the results of~\cite[Thm $1$]{cavalagli_FNE} for the two-player game.
\begin{theorem}\label{thm:symmetric_unique_symm_FNE}
Let us consider the symmetric discounted LQ game such that condition~\eqref{eq:SUFF_roots} holds. Then there always admits a unique symmetric FNE $(k_\mathrm{s}, k_\mathrm{s})$ where
    \begin{equation*}
        \displaystyle k_\mathrm{s}=- \frac{S \omega}{3} - \frac{a}{2} - \frac{\zeta}{3 S\omega},
    \end{equation*}
with 
\begin{equation}\label{eq:symm_FNE_te}
    \begin{aligned}
            \zeta &= \frac{3 a^{2}}{4} + 3 \sigma + \frac{3}{2 \gamma},\quad  \omega =  -\frac{1- \sqrt{3} i}{2} \\
            S &=
                \frac{1}{2} \sqrt[3]{\frac{27a}{\gamma}
                 + 3\sqrt{3} i
                \sqrt{
                \frac{\left(\gamma(a^{2}+4\sigma)+2\right)^{3}-27a^{2}\gamma}{\gamma^{3}}
                } }.
        \end{aligned}
\end{equation}
\end{theorem}
\begin{proof}
    Although the existence and uniqueness of a FNE have already been established in Theorem~\ref{thm:nec_multi}, the proof of the latter relies on a different argument. In particular, it is based on the relation between the optimal root $h_i^-(k_i)$ and the bounded FNE branch $r_{i, -i}^-$, which, in the symmetric setting, collapses into two functions containing both symmetric and hyperbolic components. This can be seen in Figure~\ref{fig:factorization}, where the pink curve corresponds to the bounded root $r_{i, -i}^-$ which still admits a unique solution with the map $h_i^-(k_i)$ but, in the symmetric setting, it does not represent the symmetric property for every value of $k_i$. As a result, in order to prove the theorem, a different approach is required. See Appendix~\eqref{appendix:symmetric_unique_symm_FNE} for the full proof. 
    
\end{proof}
While Theorem~\ref{thm:nec_multi} provides a necessary condition for the existence of multiple FNE, the next theorem fully characterizes the interval for which the game admits three FNE, refining the results of \cite[Theorem $2$]{Nortmann_2023_two_players} for the undiscounted game, while additionally providing closed-form expressions for the corresponding FNE. 
\begin{theorem}\label{thm:hyperbolic_FNE}
Let us consider the symmetric discounted LQ game such that condition~\eqref{eq:SUFF_roots} holds. Then the game admits two hyperbolic FNE if and only if
    \begin{equation}\label{eq:conditions_multiplicity}
        |a|> \sqrt{\sigma}+\sqrt{\sigma+\frac{1}{\gamma}},
    \end{equation}
     whose expressions are $(k_\mathrm{h_1}, k_\mathrm{h_2}), (k_\mathrm{h_2}, k_\mathrm{h_1})$ where
\begin{equation}\label{eq:hyp_FNE}
    k_{h_{1,2}}=\displaystyle \frac{1 - \gamma a^2 \pm \sqrt{(\gamma a^2 - 1)^2 - 4\gamma^2 a^2\sigma }}{2\gamma a}.
\end{equation}
\end{theorem}
\begin{proof}
    The full proof is in Appendix \eqref{appendix:hyperbolic_FNE}. 
    
\end{proof}

We remark that, while the general two-player setting may admit two or three additional equilibria, as stated in~\cite[Theorem~$2$]{Nortmann_2023_two_players} for the undiscounted game, and in Theorem~\ref{thm:nec_multi} for the discounted game, symmetry imposes a more restrictive structure on the emergence of additional equilibria. In particular, in the two-player symmetric case, additional equilibria arise in pairs, while in the more general $N$-player setting they appear according to the structured multiplicity described in~\cite{cavalagli_FNE}. For $\gamma=1$, the resulting condition~\eqref{eq:conditions_multiplicity} refines the ranges identified in~\cite[Statm. 2, Thm. 2]{Nortmann_2023_two_players} for the symmetric case.

\section{Local Properties under Best-Response Dynamics}\label{sec:BR}
In the context of FNE-seeking algorithms for LQ games, such as policy
gradient~\cite{policy_grad}, gradient descent of the best-response~\cite{grad_desce}, Lyapunov and Riccati iterations~\cite{Nortmann_iterative_methods}, and reinforcement-learning-based
approaches~\cite{offpolicy_q}, multiple equilibria may exhibit different
local behaviors depending on the iterative method employed. 
In particular,~\cite[Fig.~3]{Nortmann_iterative_methods} illustrates that, in a game
admitting three FNE, the regions of attraction associated with different Lyapunov and Riccati iterations can differ substantially, but with the central equilibrium not being attractive for any of the considered algorithms. Similarly,~\cite[Sec.~V]{grad_desce} considers the same three-equilibrium example and observes that the central FNE is a saddle point for the gradient descent dynamics. In addition to these numerical observations, \cite[Prop.~$1,2$]{Nortmann_iterative_methods} and \cite[Thm.~6]{mazumdar2019}
provide sufficient conditions for characterizing, respectively, local attractive and local saddle equilibria with respect to the considered methods.

Motivated by these considerations, this Section investigates the characterization of the local properties of a FNE with respect to the best-response dynamics. To do that, we extract the best-response law from the FOC system~\eqref{eq:best_response} to then derive the Jacobian and study the magnitude of the eigenvalues. A full analytical characterization is provided, with a subsequent numerical section. 

We emphasize that the study is algorithm-dependent and therefore depends on the particular update law used to generate the iterative dynamics. Consequently, the results presented in this Section do not necessarily extend to alternative learning schemes based on different update rules, but focus solely on the best-response dynamics.

\subsection{General Setting}

Let us consider the optimal reduced branch $k_{-i}=h_i^-(k_i)$~\eqref{eq:optimal_system} and let us compute the best-response law as an expression for $k_i$ as a function of $k_{-i}$:
\begin{equation}\label{eq:law_BR}
    k_i=(h_i^-)^{-1}(k_{-i}):=\mathrm{BR}_i(k_{-i}).
\end{equation}
Since we proved in Proposition \ref{prop:existence} that $h_i^{-\prime}(k_i)<0$ on the FNE domain, the inverse function is locally well defined. Therefore, given the best-response operator $T(k_i, k_{-i})=\begin{bmatrix}
        \mathrm{BR}_i(k_{-i})\\
        \mathrm{BR}_{-i}(k_i)
    \end{bmatrix}$, we define the best-response dynamics as
\begin{equation}\label{eq:iter_br}
    (k_i, k_{-i})^{(t+1)}=T((k_i, k_{-i})^{(t)}),
\end{equation}
whose fixed point is a FNE, and its corresponding Jacobian
\begin{equation}
    J_T(k_i, k_{-i})=\begin{bmatrix}
        0& \frac{\partial \mathrm{BR}_i(k_{-i})}{\partial k_{-i}}\\
        \frac{\partial \mathrm{BR}_{-i}(k_i)}{\partial k_i}& 0
    \end{bmatrix}.
\end{equation}
Thanks to the relation \eqref{eq:law_BR}, we have that
\begin{equation}\label{eq:br_derivative}
    \mathrm{BR}_i'(k_{-i}) = \frac{1}{h_i^{-\prime}(k_i)}<0.   
\end{equation}
From the characteristic polynomial of $J_T$, one obtains that $\lambda^2 - \frac{1}{h_i^{-\prime}(k_i)h_{-i}^{-\prime}(k_{-i})} =0$,
whose roots, namely the eigenvalues, are
\begin{equation}\label{eq:eigenvalues}
    \lambda_{\pm}=\pm \frac{1}{
    \sqrt{
    h_i^{-\prime}(k_i)
    h_{-i}^{-\prime}(k_{-i})}}.
\end{equation}
We now recall the definitions of the local properties of a fixed point of a discrete-time map in terms of the eigenvalues of its Jacobian.
\begin{definition}[Local Properties of a Fixed Point]\label{def:local_FNE}
Consider a fixed point of a two-dimensional discrete-time map.
We say that the fixed point is locally attractive if both eigenvalues of the Jacobian lie inside the unit circle, locally repelling if both
lie outside the unit circle, and a saddle if one eigenvalue lies
inside the unit circle and the other lies outside.
\end{definition}
From equation \eqref{eq:eigenvalues}, one can observe that both eigenvalues are real, have opposite signs, and have the same magnitude. From here, let us introduce
\begin{equation}\label{eq:C_br}
    C(k_i, k_{-i}):=h_i^{-\prime}(k_i)
    h_{-i}^{-\prime}(k_{-i})-1,    
\end{equation}
and derive a full analytical characterization of attractive and repelling FNE in the next theorem. 

\begin{theorem}\label{thm:br_local_stability}
Let us consider a FNE $(k_i, k_{-i})$ of a discounted LQ game such that condition~\eqref{eq:SUFF_roots} holds. Then,
with respect to the best-response iteration scheme \eqref{eq:iter_br}, the following statements hold:
\begin{align}
    C(k_i, k_{-i})>0
    &\iff
    \rho(J_T(k_i, k_{-i}))<1,
    \label{eq:C_attractive}\\
    C(k_i, k_{-i})<0
    &\iff
    \rho(J_T(k_i, k_{-i}))>1.
    \label{eq:C_repelling}
\end{align}
Consequently, a FNE satisfying $C(k_i, k_{-i})>0$ is locally
asymptotically attractive, whereas a FNE satisfying
$C(k_i, k_{-i})<0$ is locally repelling under the
best-response dynamics.
\end{theorem}
\begin{proof}
Since $h_i^{-\prime}(k_i)<0$ for both players, we have that $h_i^{-\prime}(k_i)h_{-i}^{-\prime}(k_{-i})>0$. From~\eqref{eq:eigenvalues}, we have that the spectral radius at the FNE is 
$\rho(J_T(k_i, k_{-i}))=\frac{1}{
    \sqrt{h_i^{-\prime}(k_i)h_{-i}^{-\prime}(k_{-i})}}$.
Hence, the latter is lower than one if and only if $
    h_i^{-\prime}(k_i)h_{-i}^{-\prime}(k_{-i})>1
    \iff C(k_i,k_{-i})>0$, while it is greater than one if and only if $h_i^{-\prime}(k_i)h_{-i}^{-\prime}(k_{-i})<1
    \iff C(k_i,k_{-i})<0.$ By recalling that both eigenvalues admit the same magnitude, we can either have both outside the unit circle, or both inside the unit circle. From Definition~\ref{def:local_FNE}, the claim follows. 
\end{proof}
Theorem \ref{thm:br_local_stability} not only provides an analytical characterization of each FNE with respect to the best-response dynamics, but also shows that a FNE cannot be a saddle point: whenever $C(k_i, k_{-i})>0$, both eigenvalues lie inside the unit circle, whereas for $C(k_i, k_{-i})<0$, both lie outside the unit circle. The case of $C(k_i, k_{-i})=0$ corresponds to the eigenvalues being $+1$ and $-1$, thus excluding local attractiveness. 

In the following subsection, we establish a stronger result regarding the symmetric setting, highlighting the relation between attractiveness, repulsiveness and multiplicity of FNE.

\subsection{Symmetric Setting}
From what we have proven so far, the study of the symmetric setting narrows down to two types of FNE: the symmetric one and the hyperbolic ones (when they exist). Before proving our results, we provide a useful statement regarding the symmetric FNE.
\begin{lemma}\label{lemma:symm_sqrt_d}
    Let us consider a symmetric FNE $(k_\mathrm{s}, k_\mathrm{s})$ of the symmetric discounted LQ game. Then 
    \begin{equation*}
        |a|\geq \sqrt{\sigma}+\sqrt{\sigma+\frac{1}{\gamma}}\iff |k_\mathrm{s}|\geq\sqrt{\sigma}.
    \end{equation*}
\end{lemma}
\begin{proof}
    The proof is given in Appendix~\eqref{appendix:symm_sqrt_d}. 
\end{proof}
It is worth remarking that Lemma~\ref{lemma:symm_sqrt_d}
characterizes how the coordinates of the symmetric FNE evolve as
inequality~\eqref{eq:conditions_multiplicity} is satisfied, which marks the transition from one to three FNE. This property is crucial for establishing the relation between multiplicity of FNE and the local properties of FNE with respect to the best-response dynamics, as shown in the next theorem.

\begin{theorem}\label{thm:central_FNE}
    Consider a symmetric FNE $(k_\mathrm{s}, k_\mathrm{s})$ of a discounted LQ game such that~\eqref{eq:SUFF_roots} holds. With respect to the best-response dynamics \eqref{eq:iter_br}, it is locally attractive if and only if it is unique. Conversely, it is locally repelling if and only if it is not unique. 
\end{theorem}
\begin{proof}
    The proof is given in Appendix~\eqref{appendix:central_FNE}. 
\end{proof}

Thus, in the symmetric setting, multiplicity is equivalent to the
local repulsiveness of the central symmetric FNE under the best-response dynamics. Next, we numerically illustrate the analytical results obtained for the symmetric setting and investigate whether a similar behavior persists in the general setting.

\subsection{Numerical Results}
In this section, we provide numerical evidence regarding the local attractive and repelling properties of each FNE with respect to the best-response scheme~\eqref{eq:iter_br}. Table~\ref{tab:br_convergence} reports the percentage of convergent runs of the iterative algorithm reaching each FNE under different initialization strategies. Specifically, the algorithm is initialized within an $\varepsilon$-neighborhood of each FNE, with $\varepsilon\sim \mathcal{U}([-1, 1]^2)$. We consider \(10^2\) randomly generated game configurations, and for each configuration, we generate \(10^2\) random initializations for each initialization strategy. The results are consistent with the theoretical findings and suggest that the central FNE repelling property established for the symmetric setting may persist in the general setting as well.

In relation to the numerical results of \cite{mazumdar2019}, which show that gradient-play can avoid Nash equilibria that are saddle points of the gradient dynamics, our numerical experiments exhibit a similar avoidance phenomenon under the best-response dynamics. However, unlike the non-convergent behavior observed in \cite{mazumdar2019}, where the iterates may converge to limit cycles, our experiments show that the best-response iterations converge to one of the other FNE.

\begin{table}[t]
    \centering 
    \caption{Percentage of convergent runs reaching each FNE for different initialization strategies. We consider a total of $10^2$ randomly generated game configurations, and for each configuration, $10^2$  random initializations are generated for each initialization strategy. The disturbance term is sampled from a uniform distribution $\varepsilon\sim\mathcal{U}([-1,1]^2)$.}
    \label{tab:br_convergence}
    \begin{tabular}{c c c c}
        \hline
        Initialization & $k^{(1)}$ & $k^{(2)}$ & $k^{(3)}$ \\
        \hline
        $k^{(1)} + \varepsilon$ & $98\%$ & $0\%$ & $2\%$ \\
        \hline
        $k^{(2)} + \varepsilon$ & $50.5\%$ & $0\%$ & $49.5\%$ \\
        \hline
        $k^{(3)} + \varepsilon$ & $2\%$ & $0\%$ & $98\%$ \\
        \hline
    \end{tabular}
\end{table}
In the next Section, we investigate the role of the objective function, that is the value function, together with an external metric such as the community cost.

\section{Value Function and Community Cost}\label{sec:value}
We want to investigate analytical and numerical results concerning the value function and its relation across multiple FNE. First, we analytically establish an ordering across multiple equilibria in the symmetric setting and show that the same structure persists numerically in the general case. Second, we investigate the relation between multiple FNE and external criteria such as the Community Cost, highlighting how a preferable FNE may correspond to a repelling point for the best-response method.

\subsection{Value-function ordering}
Thanks to the explicit formulas of the symmetric setting, the next theorem establishes a precise ordering among the three stable FNE, which exhibits a specific structure. 

\begin{theorem}\label{thm:value_func_order}
    Let $(k_\mathrm{s}, k_\mathrm{s}),(k_\mathrm{h_1}, k_\mathrm{h_2}), (k_\mathrm{h_2}, k_\mathrm{h_1})$ be, respectively, the symmetric and the two hyperbolic FNE of the symmetric discounted LQ game. Then, the ordering
    \begin{equation}\label{eq:ordering}
        \mathrm{min}(k_\mathrm{h_1}, k_\mathrm{h_2})<k_\mathrm{s}<\mathrm{max}(k_\mathrm{h_1}, k_\mathrm{h_2}),
    \end{equation}
    holds, and for any initial state $x_0$ and for each player $i$
    \begin{equation*}
        V_i(x_0; k_\mathrm{h_2}, k_\mathrm{h_1})\leq V_i(x_0; k_\mathrm{s}, k_\mathrm{s})\leq V_i(x_0;k_\mathrm{h_1}, k_\mathrm{h_2}).
    \end{equation*}    
\end{theorem}
\begin{proof}
    The proof is given in Appendix~\eqref{appendix:value_func_order}. 
\end{proof}
The theorem highlights the hyperbolic equilibria as unbalanced in terms of value functions: while one of the two is achieving the best outcome among the solutions, the other is achieving the worst, revealing an anti-coordination structure. On the other hand, the symmetric equilibrium presents itself as a trade-off between the two cases for both players.

\begin{figure}[t!]
    \centering
    \includegraphics[width=\linewidth]{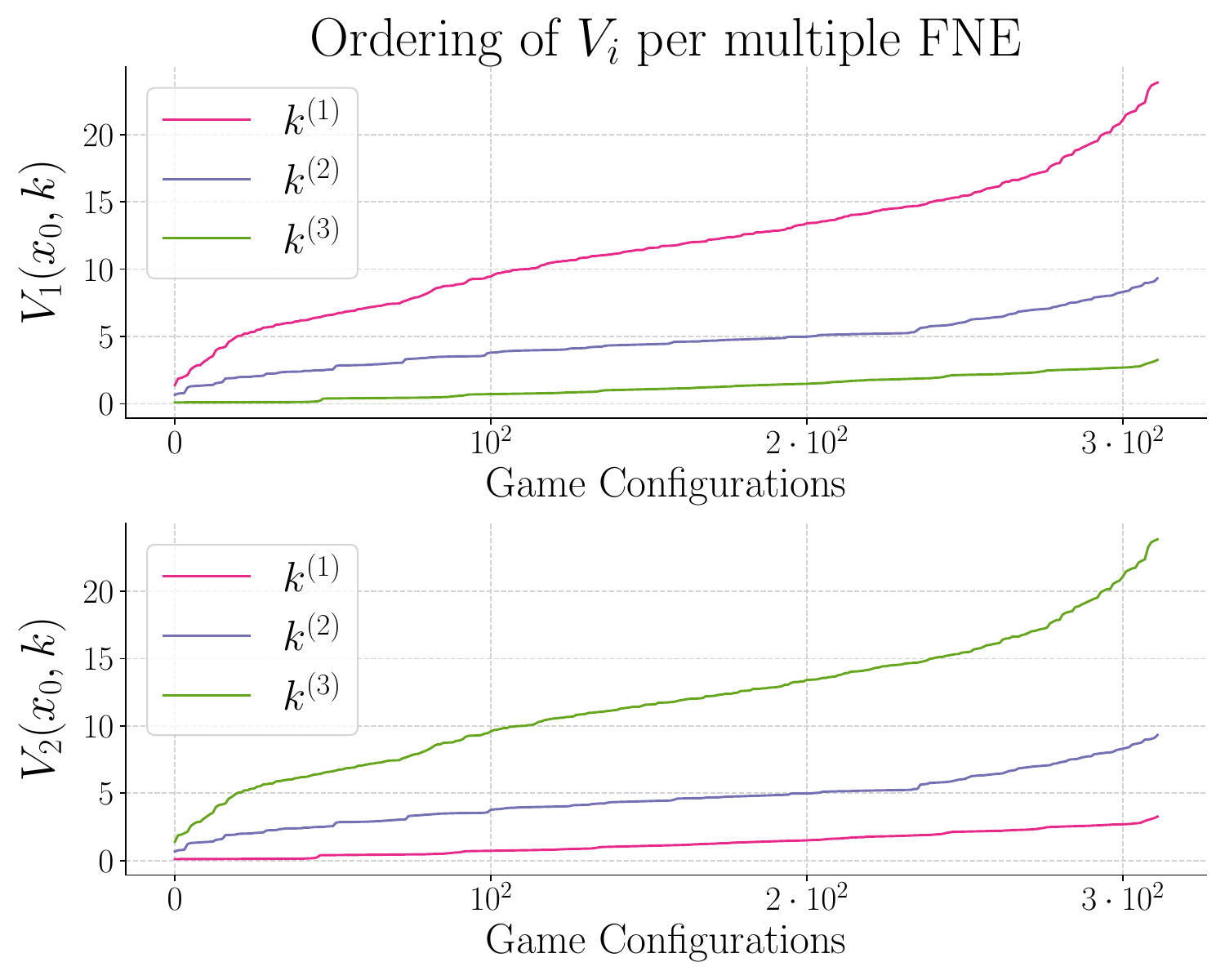}
    \caption{Display of the value functions associated with the three ordered FNE, $k^{(1)}$, $k^{(2)}$, and $k^{(3)}$, across different game configurations. As established
    theoretically in Theorem~\ref{thm:value_func_order} for the symmetric setting, the same ordering of the value functions is also observed numerically in the
    general setting.}
    \label{fig:ordering}
\end{figure}

To verify if the same structure persists in the general setting, we considered the general game, denoted the three ordered FNE by $k^{(1)}=\left(k^{(1)}_1, k^{(1)}_2\right), k^{(2)}=\left(k^{(2)}_1, k^{(2)}_2\right), k^{(3)}=\left(k^{(3)}_1, k^{(3)}_2\right)$, and computed the corresponding value functions across different game configurations. To compute the FNE, we used the numerical method provided in \cite[Algorithm $1$]{cavalagli_FNE} and let vary $\gamma\in\{0.6, 0.8, 0.9, 1\}$, $\sigma_1, \sigma_2\in[0.1, 2]$ and searched for a feasible $a$ admitting three FNE starting from the necessary condition bound~\eqref{eq:NEC_MULT} up to a maximum value of $5$. Figure~\ref{fig:ordering} shows that the value functions of player $1$ and $2$ reveal a mirrored ordering suggesting that the behavior established in the Theorem~\ref{thm:value_func_order} persists in the general setting as well. As in the symmetric case, the two outer equilibria favor one player at the expense of the other, while the central equilibrium provides a balanced outcome. However, as proved in Theorem~\ref{thm:central_FNE} for the symmetric setting and shown numerically in Section~\ref{sec:BR}, the central FNE, whenever three FNE exist, exhibits locally repelling behavior under the best-response dynamics. This makes it particularly interesting to investigate its relation with additional metrics, such as aggregate costs, as discussed in the next subsection.

\subsection{Community Cost}
Let us define the Community Cost (CC) of a given FNE $k$ as $CC(k)=V_1(x_0,k)+V_2(x_0,k)$, and consider the problem of identifying the FNE minimizing this aggregate cost. Such an aggregation of individual costs is common in network and electricity settings, where a central planner may be interested in minimizing the overall cost incurred by the agents. In this subsection, we investigate the Community Cost as an external criterion for comparing multiple FNE and examine its relation with their attractive and repulsive properties.

Using the expression of the value function in~\eqref{eq:value_function}, it can be rewritten as $V_i(x_0,k) = r_i\widetilde{V}_i(x_0,k)$, with $\widetilde{V}_i(x_0,k) = x_0^2\frac{\sigma_i+k_i^2}{1-\gamma a_{\mathrm{cl}}^2(k)}$
and therefore $CC(k)=r_1\widetilde{V}_1(x_0,k) +r_2\widetilde{V}_2(x_0,k)$. Here, $(\sigma_1,\sigma_2)$ determine the geometry of the game and, consequently, the FNE, whereas $r_1$ and $r_2$ determine their relative contribution to the aggregate cost. In the following, we therefore fix $(a,\sigma_1,\sigma_2,\gamma)$ and investigate how the community cost ordering among the resulting FNE varies with the relative scaling of the individual objectives.

Let $\lambda=r_1/r_2\in(0,\infty)$ be the relative cost weight between the two players; let $r_2=r$, and rewrite $CC(k)=r\left[\lambda\widetilde{V}_1(x_0,k)+\widetilde{V}_2(x_0,k)\right]$. Since $r>0$ and multiplies the whole sum, the minimizer of $CC$ depends only on $\lambda$. Figure~\ref{fig:cc} displays the community cost curves associated with the three ordered FNE $k^{(1)},k^{(2)},k^{(3)}$ for the game configuration
$(a,\sigma_1,\sigma_2,\gamma)=(3,1,1,1)$ but the same structure is observed numerically across multiple tested game configurations. The intersections of the curves identify the values
$\lambda_1^\star$ and $\lambda_2^\star$ at which the community cost minimizing equilibrium switches. Accordingly, the three intervals $(0,\lambda_1^\star),\ (\lambda_1^\star,\lambda_2^\star),\ (\lambda_2^\star,\infty)$ correspond to the regions in which $k^{(1)}$, $k^{(2)}$, and $k^{(3)}$, respectively, minimize the aggregate cost. For small or large values of $\lambda$, the community cost places greater weight on one of the two players, favoring the FNE that yields the lowest value function for the more heavily weighted player. For intermediate values of $\lambda$, the central FNE minimizes the aggregate cost. This observation, combined with the local properties established in the previous section, highlights a possible mismatch between the equilibria selected by iterative methods and those minimizing the community cost.

We emphasize that this analysis is qualitative and that the community cost metric depends on the relative scaling $r_i$ of the individual objectives. Therefore, the designation of an FNE as community cost minimizing is conditional on the chosen objective scaling and should not be interpreted as an intrinsic property of the equilibrium landscape.

\begin{figure}[t!]
    \centering
    \includegraphics[width=\linewidth]{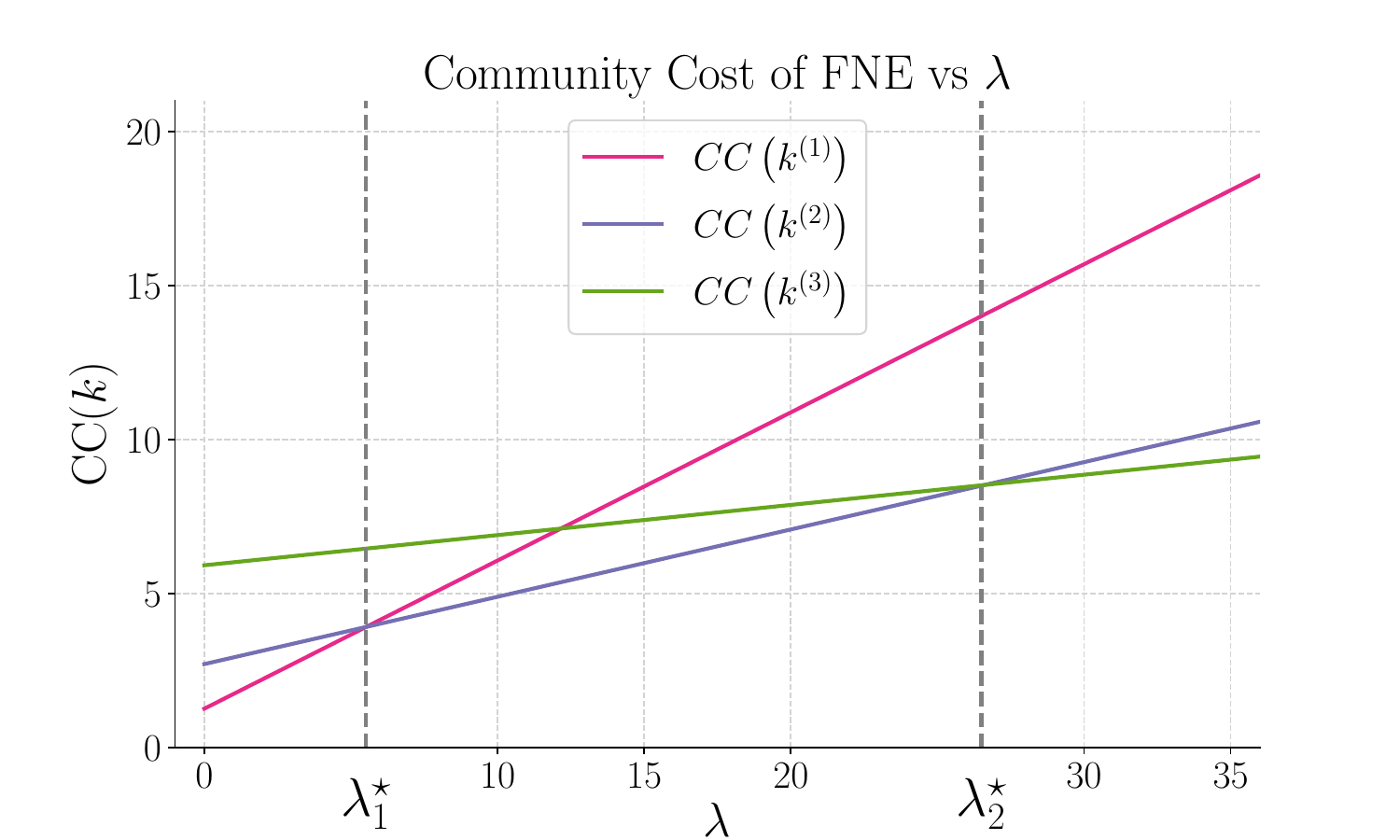}
    \caption{Community Cost of the three FNE as a function of $\lambda$. The black dashed vertical lines correspond, from left to right, to $\lambda_1^\star$ and $\lambda_2^\star$, respectively, indicating the values of $\lambda$ at which the Community-Cost-minimizing FNE switches from one equilibrium to another. }
    \label{fig:cc}
\end{figure}
\section{Conclusions}
In this paper, we considered the class of
infinite-horizon discounted LQ games, with a focus on the two-player
scalar setting. We showed that discounting introduces a class of
finite-cost FNE that does not necessarily stabilize the system,
thereby extending the analysis of undiscounted LQ games and recovering
the latter as the special case $\gamma=1$. We provided a sufficient
condition under which every FNE is stabilizing, together with
conditions for the existence and uniqueness of FNE, and a necessary
condition for the existence of multiple FNE. The local attractive and
repelling properties of FNE under the best-response dynamics were
fully characterized, and, in the symmetric setting, were
shown to be directly related to the multiplicity of FNE. In the
presence of multiple FNE, analytical and numerical results showed that
the value functions satisfy a specific ordering. External qualitative
metrics, such as community cost, were also considered, with numerical
results illustrating the relation between the preferred FNE and its
local attractive and repelling properties.

This paper recovers the results of~\cite{cavalagli_FNE} for the two-player setting and extends the results of~\cite{Nortmann_2023_two_players, feedbackNE} to the discounted case, providing further characterizations of multiplicity and closed-form solutions, particularly under symmetry. Extensions to multidimensional games might require a different approach, since the analysis relies on a closed-form expression of the value function. Nevertheless, the scalar setting provides a useful benchmark for more complex scenarios, such as finite-horizon LQ settings and nonlinear settings where closed-form solutions are generally unavailable.

\appendix
\section{Proofs}   
\subsection{Proof of Lemma~\ref{lemma:exclude_null}}
\label{appendix:exclude_null}
$(\Rightarrow)$. Let us assume $k_i=0$ so that system~\eqref{eq:best_response} becomes respectively $f_i(0, k_{-i})=\gamma \sigma_{i} \left(a + k_{-i}\right)=0$ and $f_{-i}(0, k_{-i})=a \gamma \sigma_{-i} - a \gamma k_{-i}^{2}+  k_{-i} \left(- a^{2} \gamma + \gamma \sigma_{-i} + 1\right)=0$. The first equation implies $k_{-i}=-a$ as the only admissible solution (considering that $\gamma\sigma_i\neq0$). By substituting it inside the second equation we obtain $a=0$.\\
$(\Leftarrow)$. Let us assume $a=0$ and substitute it inside~\eqref{eq:best_response}
\begin{equation}\label{eq:equa_nulls}
    \begin{cases}
        k_{-i}\left(\gamma\sigma_i-\gamma k_i^2\right) + k_i \left(\gamma \sigma_{i} - \gamma k_{-i}^{2} + 1\right)=0,\\
        k_i\left(\gamma\sigma_{-i}-\gamma k_{-i}^2\right) + k_{-i} \left(\gamma \sigma_{-i} - \gamma k_i^{2} + 1\right)=0,\\
        \left|k_i+k_{-i}\right|<1.
    \end{cases}
\end{equation}
Let us subtract the first two equations and obtain 
\begin{equation}\label{eq:subtract}
     k_i\left( \gamma \sigma_{i} - \gamma \sigma_{-i} + 1\right) + k_{-i} \left(\gamma \sigma_{i} - \gamma \sigma_{-i} - 1\right)=0.
\end{equation}
If $(\gamma \sigma_{i} - \gamma \sigma_{-i} + 1)=0$ or $(\gamma \sigma_{i} - \gamma \sigma_{-i} - 1)=0$ then the equation admits as the only solution $k_i=0$ or $k_{-i}=0$. As a consequence, by substituting it inside the system~\eqref{eq:equa_nulls} we obtain as a final FNE the null vector $(0, 0)$ which is trivially a stable FNE and so the claim is proven. Thus, assume $(\gamma \sigma_{i} - \gamma \sigma_{-i} + 1)\neq 0$, $(\gamma \sigma_{i} - \gamma \sigma_{-i} - 1)\neq 0$, denote $m=\frac{ (\gamma \sigma_{i} - \gamma \sigma_{-i} + 1)}{(- \gamma \sigma_{i} + \gamma \sigma_{-i} + 1)}$, and solve equation~\eqref{eq:subtract} w.r.t. $k_{-i}$, which gives $k_{-i}= k_i\cdot m$. Let us substitute the latter into the first two equations of the system~\eqref{eq:equa_nulls} to obtain
\begin{equation}\label{eq:sys_null}
\begin{cases}
     k_i \left(\gamma m^{2} k_i^{2} - \gamma m \sigma_{i} + \gamma m k_i^{2} - \gamma \sigma_{i} - 1\right)=0,\\
      k_i \left(\gamma m^{2} k_i^{2} - \gamma m \sigma_{-i} + \gamma m k_i^{2} - \gamma \sigma_{-i} - m\right)=0.
\end{cases}
\end{equation}
If $k_i=0$ the claim is obtained. If $k_i\neq0$ then let us solve both equations w.r.t. $k_i^2$ and obtain $k_i^2=\displaystyle\frac{\gamma \sigma_i(m+1)+1}{\gamma m(m+1)}$ from the first one, and $k_i^2=\displaystyle \frac{\gamma \sigma_{-i}(m+1)+m}{\gamma m(m+1)}$ from the second one. By substituting $m$ in both of them and subtracting the equations, we obtain $\frac{\left(\sigma_{i} - \sigma_{-i} \left(\gamma \sigma_{i} - \gamma \sigma_{-i} + 1\right)\right) \left(- \gamma \sigma_{i} + \gamma \sigma_{-i} + 1\right)}{\gamma \sigma_{i} - \gamma \sigma_{-i} + 1}=0$, where the denominator and the second term of the numerator are non zero by assumption. Then, the equation is true if and only if $\left(\sigma_{i} - \sigma_{-i} \left(\gamma \sigma_{i} - \gamma \sigma_{-i} + 1\right)\right)=0$, which implies $\sigma_i=\sigma_{-i}$ and so $k_i=k_{-i}$. By substituting into the original equation of the system~\eqref{eq:equa_nulls} we obtain a one-variable equation $k_i \left(- 2 \gamma \sigma_{i} + 2 \gamma k_i^{2} - 1\right)=0$ which admits as the only candidate solution $k_i=- \frac{\sqrt{4 \sigma_{i} + \frac{2}{\gamma}}}{2}$, which does not satisfy the stability condition as $|a+k_i+k_{-i}|=\left|2k_i\right|=\left| \sqrt{4 \sigma_{i} + \frac{2}{\gamma}}\right|<1\ \iff\ \sigma_i<\frac{1}{4}-\frac{1}{2\gamma}<0$. As a consequence, $k_i=0$ is the only solution solving system~\eqref{eq:sys_null} for all discussed cases.

%


\subsection{Proof of Proposition~\ref{prop:existence}}\label{appendix:existence}
Let us take the expression $h_i^-(k_i)$ from~\eqref{eq:roots}, isolate $-a$ from the first fraction and then multiply the rest with $\frac{\gamma \left(\sigma_{i} - k_i^{2}\right) + \sqrt{\gamma \Delta(k_i, \sigma_i)}}{\gamma \left(\sigma_{i} - k_i^{2}\right) + \sqrt{\gamma \Delta(k_i, \sigma_i)}}$ which yields its rationalized version
\begin{equation*}
    h_i^-(k_i)=- a - \frac{2 k_i \left(\gamma \sigma_{i} + 1\right)}{\gamma \left(\sigma_{i} - k_i^{2}\right) + \sqrt{\gamma \Delta(k_i, \sigma_i)}}.
\end{equation*}
In the limit for $k_i\rightarrow 0$ we directly obtain $h_i^-(0)=-a$. Now let us focus on the roots of the expression $h_i^-(k_i)=0$. To deal with that we take the latter expression and set the square root on one side of the equation and the rest to the other side $\gamma \left(- 2 a k_i + \sigma_{i} - k_i^{2}\right) =\sqrt{\gamma \Delta(k_i, \sigma_i)}$, 
with the aim to square both sides. When squaring up, we must take into account the positivity of the left side which is ensured if
\begin{equation}\label{eq:dom_positivity}
    k_i\in\left[-a-\sqrt{a^2+\sigma_i}, -a+\sqrt{a^2+\sigma_i}\right].
\end{equation}
Through algebraic manipulations one can obtain the squared expression $\gamma a k_i^2+k_i\left(\gamma a^2 -\gamma \sigma_i-1\right)-\gamma a\sigma_i=0$, which is a quadratic in $k_i$ with two nonzero roots 
\begin{equation}\label{eq:intersection}
    \overline{k}_\pm(\sigma_i)= -a + \frac{n_i \pm \sqrt{\delta_i}}{2 \gamma a },
\end{equation}
with $\delta_i=n_i^2-4\gamma a^2>0$ and $n_i=\gamma\left( a^2+ \sigma_i\right)+1>0$. Only one of the two satisfies the positivity of the right side of the previous equation i.e. condition~\eqref{eq:dom_positivity}. In fact, we will prove next that $\left|\overline{k}_+(\sigma_i)+a\right|> \sqrt{a^2+\sigma_i}$ and $\left|\overline{k}_-(\sigma_i)+a\right|\leq \sqrt{a^2+\sigma_i}$, which is true for every $a$ but we will assume for simplicity $a>0$ (the proof for the opposite case is symmetric). Let us first compute $\overline{k}_+(\sigma_i)+a=\frac{n_i+\sqrt{\delta_i}}{2\gamma a}$ and multiply both sides by $2\gamma a>0$, isolate $\sqrt{\delta_i}$ on one side and then safely square as both sides are trivially positive. What we've obtain is $\delta_i > \left(2\gamma a \sqrt{a^2+\sigma_i}-n_i\right)^2$,
which excludes $\overline{k}_+(\sigma_i)$ as an admissible root.
Now let us take $\overline{k}_-(\sigma_i)+a=\frac{n_i-\sqrt{\delta_i}}{2\gamma a}$ and rationalize it by multiplying numerator and denominator by the positive quantity $n_i+\sqrt{\delta_i}$ yielding $\overline{k}_-(\sigma_i)+a=\frac{2a}{n_i+\sqrt{\delta_i}}$. Then let us take the inequality $\frac{n_i-\sqrt{\delta_i}}{2\gamma a}\leq \sqrt{a^2+\sigma_i}$, multiply both sides by $2\gamma a>0$ and safely square. Through simple manipulations one obtains $n_i-2\gamma a^2 - \sqrt{\delta_i}\leq 0$, which proves $\overline{k}_-(\sigma_i)$ as an admissible root for the equation $h_i^-(k_i)=0$.

\subsection{Proof of Lemma~\ref{lemma:facto_FNE_law}}
\label{appendix:facto_FNE_law}
Let us take the FNE law~\eqref{eq:FNE_law}, which is a quadratic in both variables $k_i, k_{-i}$, fix $k_i$ and solve it w.r.t. $k_{-i}$: $p_{i, -i}(k_i, k_{-i})=(k_{-i}-r_{i,-i}^-(k_i))(k_{-i}-r_{i,-i}^+(k_i))$, with:
\begin{equation*}
    r_{i, -i}^\pm(k_i)= \frac{\sigma_i + k_i^{2} \pm \sqrt{\nabla(k_i, \sigma_i, \sigma_{-i}) }}{2 k_i}.
\end{equation*}
From here we can categorize the two branches as \emph{unbounded} ($+$) and \emph{bounded} ($-$). We will prove the boundedness for $r_{i,-i}^-(k_i)$ and unboundedness for $r_{i,-i}^+(k_i)$. 

$(-)$. Let us rewrite the function $r_{i,-i}^-(k_i)$ by multiplying the numerator by the positive quantity $1=\frac{\sigma_{i} + k_i^{2} + \sqrt{\nabla(k_i, \sigma_i, \sigma_{-i})}}{\sigma_{i} + k_i^{2} + \sqrt{\nabla(k_i, \sigma_i, \sigma_{-i})}}$ and obtain the rationalized function
\begin{equation*}
    r_{i,-i}^-(k_i) = \frac{2\sigma_{-i}k_i}{\sigma_{i} + k_i^{2} + \sqrt{\nabla(k_i, \sigma_i, \sigma_{-i}) }}.
\end{equation*}
Let us first observe that $\displaystyle\lim_{k_i\longrightarrow0}r_{i, -i}^-(k_i)=0$. While for $k\longrightarrow\infty$, the denominator's degree dominates over the numerator's; as a consequence we have $\displaystyle\lim_{k_i\rightarrow\pm\infty}r_{i,-i}^-=0$. Therefore, the function is bounded. The function admits as stationary points $k_i=\pm\sqrt{\sigma_i}$ whose corresponding values are $r_{i,-i}^-(\pm\sqrt{\sigma_i})=\pm \left(\sqrt{\sigma_{i}}- \sqrt{\sigma_i-\sigma_{-i}}\right)\leq \pm \sqrt{\sigma_i}$, which proves boundedness.

$(+)$. Let us take $r_{i,-i}^+(k_i)$ and expand each term
\begin{equation*}
    r_{i,-i}^+(k_i) = \frac{\sigma_i}{2k_i}+k_i\left[\frac{1}{2}+\frac{1}{2}\sqrt{\frac{\sigma_i^2}{k_i^2}+k_i^2+2\sigma_i-4\sigma_{-i}}\right].
\end{equation*}
When approaching infinity and zero, the respective multiplying term $k_i$ and the fraction $\frac{\sigma_i}{2 k_i}$ determine the sign of the limit; As a consequence, we have $\displaystyle \lim_{k_i\rightarrow\pm\infty}r_{i,-i}^+=\pm\infty$ and $\displaystyle\lim_{k_i\rightarrow 0^\pm}r_{i,-i}^+=\pm\infty$. Therefore, the function is unbounded. The stationary points are $k_i=\pm\sqrt{\sigma_i}$ whose corresponding values are $r_{i,-i}^+(\pm\sqrt{\sigma_i})=\pm \left(\sqrt{\sigma_{i}}+ \sqrt{\sigma_i-\sigma_{-i}}\right)\gtrless \pm \sqrt{\sigma_i}$. That means that the function never crosses the $k_i$ axis as the extremes are always concordant in sign with the half-space, proving the unboundedness. In particular, we have \begin{equation}\label{eq:branches_FNE}
    \left|r_{i,-i}^+(k_i, k_{-i})\right|\geq \sqrt{\sigma_i}+\sqrt{\sigma_i-\sigma_{-i}}.
\end{equation}
The classification is true for any $i$ such that $\sigma_i>\sigma_{-i}$; if the latter condition is not true, we simply take the FNE law~\eqref{eq:FNE_law}, fix $k_{-i}$, solve it in terms of $k_i$ and follow the same procedure.

\subsection{Proof of Theorem \ref{thm:nec_multi}}\label{appendix:suff_multi}
\emph{Existence and uniqueness}. Let us start with the existence proof and assume throughout the remainder of the proof $a>0$ as the proof is symmetric for $a<0$. Consider the $-i$-th reduced optimal equation from system~\eqref{eq:optimal_system} and the $-i$-th factorization from the FNE law~\eqref{eq:FNE_law_factorized} whose intersection leads to a single one-variable functional $G(k_i)=r_{i,-i}^-(k_i)-h_i^-(k_i)$ of expression
\begin{equation*}
    G(k_i)= a+k_i+ \frac{\sqrt{\gamma \Delta(k_i, \sigma_i)}-\gamma \sqrt{\nabla(k_i, \sigma_i, \sigma_{-i})}}{2\gamma k_i}=0,
\end{equation*}
with $\nabla(k_i, \sigma_i, \sigma_{-i})>0$ being positive thanks to the ordering $\sigma_i>\sigma_{-i}$, and remarking the identity $\nabla(k_i, \sigma_i, \sigma_{-i})=(\sigma_i-k_i^2)^2 + 4 k_i^2 (\sigma_i-\sigma_{-i})$ which will be used later. The proof narrows down to the study inside the domain~\eqref{eq:sFNE_dom} and subsequently the usage of the Intermediate Value Theorem (IVT). Let us start with the study of the sign of the function as $k_i\longrightarrow0$ as it is a singularity point, and let us consider the fraction $\frac{\sqrt{\gamma \Delta(k_i, \sigma_i)}-\gamma \sqrt{\nabla(k_i, \sigma_i, \sigma_{-i})}}{2\gamma k_i}$, multiply it by $1=\frac{\sqrt{\gamma \Delta(k_i, \sigma_i)}+\gamma \sqrt{\nabla(k_i, \sigma_i, \sigma_{-i})}}{\sqrt{\gamma \Delta(k_i, \sigma_i)}+\gamma \sqrt{\nabla(k_i, \sigma_i, \sigma_{-i})}}$ and obtain after few algebraic manipulations the rationalized form of $2k_i\frac{\gamma \sigma_{-i}+1}{\sqrt{\gamma\Delta(k_i, \sigma_i)}+\gamma\sqrt{\nabla(k_i, \sigma_i, \sigma_{-i})}}$. By substituting the latter inside the expression of $G(k_i)$, we obtain $G(k_i)=a+k_i\left(1+\frac{2(\gamma \sigma_{-i}+1)}{\sqrt{\gamma\Delta(k_i, \sigma_i)}+\gamma\sqrt{\nabla(k_i, \sigma_i, \sigma_{-i})}}\right)$ which converges to $a>0$ as $k_i\longrightarrow0$. Now, by evaluating $G(k_i)$ at $k_i=-a$ we obtain $\frac{\gamma\sqrt{\nabla(-a, \sigma_i, \sigma_{-i})} - \sqrt{\gamma\Delta(-a, \sigma_i)}}{2 a \gamma}$, which is always negative as $\Delta(-a, \sigma_i)>\gamma \nabla(-a, \sigma_i, \sigma_{-i})\ \forall\sigma_{i,-i}$. Therefore, by the IVT we have
\begin{equation*}
    \begin{cases}
        G(0)>0\\
        G(-a)<0   \end{cases}\Rightarrow\exists\  k_i^\star\in(-a,0)\ |\ G(k_i^\star)=0,
\end{equation*}
which proves existence. To prove uniqueness, we show that $G(k_i)$ is strictly increasing. Let us compute the first derivative
\begin{equation}\label{eq:F_derivative}
    G^\prime(k_i)= 1 + \frac{\sigma_i^2- k_i^4}{2k_i^2}
    \left(
    \frac{1}{\sqrt{\nabla(k_i,\sigma_i,\sigma_{-i})}}
    - \frac{\sqrt{\gamma}}
    {\sqrt{\Delta(k_i,\sigma_i)}}
    \right),
\end{equation}
and observe that $\frac{1}{\sqrt{\nabla(k_i,\sigma_i,\sigma_{-i})}}
    > \frac{\sqrt{\gamma}}
    {\sqrt{\Delta(k_i,\sigma_i)}}>0$. Thanks to that, we can first state that if $|k_i|\leq \sqrt{\sigma_i}$ then the derivative is the sum of positive terms and so it is positive. If $|k_i|>\sqrt{\sigma_i}$ then let us observe that $\nabla(k_i, \sigma_i, \sigma_{-i})\geq \left(\sigma_i-k_i^2\right)^2$ if and only if $1/\sqrt{\nabla(k_i, \sigma_i, \sigma_{-i}}\leq 1/\left(k_i^2-\sigma_i\right)$. By using that information, we obtain a lower bound 
    \begin{equation*}
        G^\prime(k_i)>1-\frac{k_i^2+\sigma_i}{2k_i^2}>0,
    \end{equation*}
    which proves strict positivity for every value of $k_i\neq0$. Consequently, the solution $k_i^\star\in(-a, 0)$ to $G(k_i^\star)=0$ is unique.
    
    \emph{Multiplicity}.  So far we have proven that the intersection with the $r_{i,-i}^-(k_i)$ branch always exists and is unique. In order for the game to admit multiple equilibria, the remaining branch $r_{i,-i}^+(k_i)$ must also be intersected. We know from Lemma~\ref{lemma:facto_FNE_law} that the maximum is attained in $\mathbb{R}^-$ and the minimum in $\mathbb{R}_{\geq 0}$ with their corresponding values expressions in~\eqref{eq:branches_FNE}. In order for the branch to admit FNE it must be inside the FNE square~\eqref{eq:sFNE_dom}, and that means the maxima must exceed $-a$, or equivalently, the minimum must be lower than $-a$, see Figure~\ref{fig:suff_cond} for an example. The claim follows. .

\subsection{Proof of Theorem \ref{thm:symmetric_unique_symm_FNE}}\label{appendix:symmetric_unique_symm_FNE}
Let us solve the symmetric branch equation \eqref{eq:symm_branch} w.r.t. $a$ and obtain $a=W(k_i)$, where
\begin{equation}\label{eq:W_functional_a}
W(k_i)= \frac{\gamma \left(\sigma - 3 k_i^{2}\right) - \sqrt{\gamma \Delta(k_i, \sigma_i)}}{2 \gamma k_i}.
\end{equation}
We first show that $W(k_i)$ is strictly monotone, implying that $a=W(k_i)$ admits at most one solution. We then characterize the range of  $W(k_i)$ over the domain of interest \eqref{eq:sFNE_dom} to establish existence. Let us compute the first derivative of $W(k_i)$ that is $W^\prime(k_i)=\frac{\sqrt{\gamma \Delta(k_i, \sigma_i)}(\sigma^2-k_i^4)-\Delta(k_i, \sigma_i) \left(\sigma + 3k_i^2\right)}{2k_i^2\Delta(k_i, \sigma_i)}$. Considering that the denominator is strictly positive, the sign of the function depends on the numerator. If $|k_i|\geq\sqrt{\sigma}$ then the numerator is a sum of negative terms and so the overall derivative is negative. If $|k_i|<\sqrt{\sigma}$ then let us use the property $\sqrt{a+b}\leq\sqrt{a}+\frac{b}{2\sqrt{a}}$ for every $a,b$ to find an upper bound on $\sqrt{\gamma \Delta(k_i, \sigma_i)}\leq \gamma (\sigma+k_i^2)+\frac{2k_i^2}{\sigma+k_i^2}$. By using the latter inequality inside the numerator, and by sub-sequentially substituting the expression of $\Delta(k_i, \sigma_i)$, we get
\begin{equation*}
    W^\prime(k_i) \leq  \frac{- \left(2\gamma k_i^4 +k_i^2\left(7+4\gamma\sigma\right)+\sigma\left(2\gamma\sigma+1\right)\right)}{\Delta(k_i, \sigma_i)}<0.
\end{equation*}
Thus, the function $W(k_i)$ is strictly monotone for all $k_i \in \mathbb{R}_{\neq 0}$ and consequently the equation $W(k_i)=a$ admits at most one solution. To establish existence and show that the solution belongs to the domain~\eqref{eq:sFNE_dom}, we characterize the range of $W(k_i)$ over $k_i\in(-a, 0)$, by assuming $a>0$ as the proof is symmetric for $a<0$. Let us consider $k_i=-a$ and study the inequality $W(-a)>a$, which after few algebraic manipulations leads to $\frac{\gamma a^2 - \gamma \sigma_i +\sqrt{\gamma \Delta(-a, \sigma_i)}}{2\gamma a}>0$. Through the isolation of the square root on one side and then squaring up, one obtains $4\gamma a^2\left(\sigma +1\right)>0$ proving that $W(-a)>a$. Let us now study the limit of $W(k_i)$ at $0$ as it is a singularity point and rationalize the function by multiplying the expression of $W(k_i)$ by $1=\frac{\gamma (\sigma-3k_i^2)+\sqrt{\gamma \Delta(k_i, \sigma_i)}}{\gamma (\sigma-3k_i^2)+\sqrt{\gamma \Delta(k_i, \sigma_i)}}$, leading to the final form of $W(k_i)=k_i\frac{2\gamma k_i^2-2\gamma \sigma - 1}{\gamma(\sigma - 3k_i^2) + \sqrt{\gamma \Delta(k_i, \sigma_i)}}$, which converges to $0$ when $k_i\longrightarrow 0$. Therefore, $\lim_{k_i\longrightarrow0}W(k_i)<a<W(-a)$. Since $W(k_i)$ is continuous on the domain~\eqref{eq:sFNE_dom}, and it is strictly monotone, it follows that there exists a unique $k_s^\star\in(-a, 0)$ such that $W(k_s^\star)=a$. The corresponding closed-form expression for $k_s^\star$ is obtained through Cardano's formula by solving the cubic equation resulting from the squared form of~\eqref{eq:symm_branch} that is
\begin{equation}\label{eq:cubic}
    \mathcal{C}(k_i) = 2 \gamma k_i^{3} + 3 a \gamma k_i^{2}+ k_i \left(a^{2} \gamma - 2 \sigma \gamma - 1\right)- a \sigma \gamma.
\end{equation}

\subsection{Proof of Theorem \ref{thm:hyperbolic_FNE}}\label{appendix:hyperbolic_FNE}
 Let us take the hyperbolic expression~\eqref{eq:hyperb_branch}, isolate the square root on one side
    \begin{equation}\label{eq:hyp}
   \sqrt{\gamma \Delta(k_i, \sigma)} =  -\sigma \gamma - \gamma k_i \left(2 a + k_i\right),
\end{equation}
    and subsequently square both terms. Before that we first define the interval of positivity of the right side of~\eqref{eq:hyp} which turns out to be an interval centered at $-a$ of width $2\cdot\sqrt{a^2-\sigma}$; by intersecting the latter with the known domain~\eqref{eq:sFNE_dom}, we obtain the following domain for which we have hyperbolic solutions:  
\begin{equation}\label{dom:hyperb_dom}
    \text{dom}\left(\mathcal{D}(k_i)\right)=
    \begin{cases}
    [-a, -a+\sqrt{a^2-\sigma}] & a\geq\sqrt{\sigma}\\
    [ -a-\sqrt{a^2-\sigma}, -a]& a\leq-\sqrt{\sigma}\\
    \emptyset & |a|< \sqrt{\sigma}
\end{cases},
\end{equation}
where $\mathcal{D}(k_i)$ is the squared polynomial~\eqref{eq:hyp}, i.e.:
\begin{equation}\label{eq:quad_hyp}
    \mathcal{D}(k_i) = -4\gamma k_i\left(\gamma a k_i^2 + k_i\left(\gamma a^2-1\right)+\gamma \sigma a \right).
\end{equation}
The main expression, as we do not consider $(4\gamma k_i)\neq 0$, is a quadratic in $k_i$ of solutions $k_\mathrm{h_{1,2}}=\frac{(1-\gamma a^2)\pm \sqrt{\mathrm{discr}(\mathcal{D})}}{2\gamma a}$, with $\mathrm{discr}(\mathcal{D})=(a^{2}\gamma - 1)^{2} - 4a^{2}\sigma\gamma^{2}$ the discriminant. The latter is a quartic function in $a$
with the following ranges of positivity
\begin{equation}\label{eq:discri}
         a\leq a_1 \ \vee a_2\leq a \leq a_3\ \vee a\geq a_4,
\end{equation}
where $a_i=\pm\sqrt{\sigma+\frac{1}{\gamma}}\mp\sqrt{\sigma}
$ are the zeros of $\mathrm{discr}(\mathcal{D})$, as derived in~\cite[Eq. $19$]{Nortmann_2023_two_players} for the undiscounted case, and they are ordered according to the following sign assignments$(-,-)$, $(-,+)$, $ (+,-)$,$(+,+)$. As a consequence, the couple of hyperbolic roots $k_\mathrm{h_{1,2}}$ satisfying equation~\eqref{eq:quad_hyp} are real if the discriminant is positive, which is encoded in \eqref{eq:discri}.

Now, in order for the real hyperbolic roots to be a FNE, we must intersect the latter range of values with the domain~\eqref{dom:hyperb_dom}. To do that, we reverse the condition: find the intervals of $a$ for which the hyperbolic expression does not admit a FNE, i.e., the second term of~\eqref{eq:hyp} is negative. Let us take the domain expression $ -\sigma \gamma - \gamma k_i \left(2 a + k_i\right)$, substitute the solutions of the hyperbolic branch $k_\mathrm{h_1}, k_\mathrm{h_2}$ and verify for which value of $a$ it is not positive:
\begin{equation*}
    \begin{aligned}
        k_\mathrm{h_1}:\ \ \frac{- a^{4} \gamma^{2} +1  -\sqrt{(\gamma a^2 - 1)^2 - 4\gamma^2 a^2\sigma}(a^2\gamma + 1) }{2 a^{2} \gamma},\\
        k_\mathrm{h_2}:\ \ \frac{- a^{4} \gamma^{2} +1  +\sqrt{(\gamma a^2 - 1)^2 - 4\gamma^2 a^2\sigma}(a^2\gamma + 1) }{2 a^{2} \gamma}.
    \end{aligned}
\end{equation*}
Because both expressions are even w.r.t. $a$, we obtain that they are both negative for $|a|\leq\frac{1}{\sqrt{\gamma}}$. By intersecting the latter with the range of existence shown before~\eqref{eq:discri}, we exclude the middle interval $(a_2,a_3)$. In fact, one can observe that, by rewriting $a_3$ as $\sqrt{\frac{\gamma\sigma +1}{\gamma}}-\sqrt{\sigma}$, then the inequality $a_3=\sqrt{\frac{\gamma\sigma +1}{\gamma}}-\sqrt{\sigma}<\frac{1}{\sqrt{\gamma}}$ is ensured for every parameter choice. Considering that $a_2=-a_3$, the symmetric claim follows. Lastly, we exclude the cases of the discriminant being null, i.e., $a=a_{1,4}$, as they correspond to the symmetric solution (see Lemma~\ref{lemma:symm_sqrt_d}) so they do not provide an additional FNE.

\subsection{Proof of Lemma~\ref{lemma:symm_sqrt_d}}
\label{appendix:symm_sqrt_d}
Let us take the symmetric branch~\eqref{eq:symm_branch} and solve it with respect to $a$:
$W(k_i)=a$, where $W(k_i)$~\eqref{eq:W_functional_a} has been proven to be a strictly decreasing function in Theorem~\ref{thm:symmetric_unique_symm_FNE}, therefore yielding a unique $k_\mathrm{s}$ satisfying $W(k_\mathrm{s})=a$. Let us assume $a>0$ and therefore $k_\mathrm{s}<0$. Then, by the strict monotonicity of $W(k_i)$, $k_\mathrm{s}\leq-\sqrt{\sigma} \iff
W(k_\mathrm{s})\geq W(-\sqrt{\sigma})=\sqrt{\sigma}+\sqrt{\sigma+\frac{1}{\gamma}}.$ Considering that $W(k_\mathrm{s})=a$, this proves the claim for $a>0$. The case $a<0$ follows similarly by symmetry.

\subsection{Proof of Theorem \ref{thm:central_FNE}}\label{appendix:central_FNE}
    Let us consider the FNE $(k_\mathrm{s}, k_\mathrm{s})$ of which we know from Theorem~\ref{thm:br_local_stability} that it is locally attractive if and only if $C(k_\mathrm{s}, k_\mathrm{s})$ is positive. Let us unpack its definition from~\eqref{eq:C_br} by substituting the algebraic expressions: $C(k_\mathrm{s})= \left(\sigma - k_\mathrm{s}^{2}\right) D(k_\mathrm{s})$ with $D(k_\mathrm{s}) = -(\sigma+k_\mathrm{s}^2)^2 \sqrt{\gamma \Delta(k_\mathrm{s}, \sigma)}+ \gamma (\sigma+k_\mathrm{s}^2)^3 +2k_\mathrm{s}\left(\sigma+3k_\mathrm{s}^2\right)$. Let us focus on the term $D(k_\mathrm{s})$ and factor $(\sigma+k_\mathrm{s}^2)^2$: $D(k_\mathrm{s})=(\sigma+k_\mathrm{s}^2)^2\left[\gamma (\sigma+k_\mathrm{s}^2) -\sqrt{\gamma \Delta(k_\mathrm{s}, \sigma)} \right]+2k_\mathrm{s}\left(\sigma+3k_\mathrm{s}^2\right)$. Let us take the first multiplication and rationalize the term $\gamma (\sigma+k_\mathrm{s})^2 -\sqrt{\gamma \Delta(k_\mathrm{s}, \sigma)}$ by multiplying $1=\frac{\gamma (\sigma+k_\mathrm{s})^2 +\sqrt{\gamma \Delta(k_\mathrm{s}, \sigma)} }{\gamma (\sigma+k_\mathrm{s})^2 +\sqrt{\gamma \Delta(k_\mathrm{s}, \sigma)} }$ which yields to $-\frac{4\gamma k_\mathrm{s}^2}{\gamma (\sigma+k_\mathrm{s})^2 +\sqrt{\gamma \Delta(k_\mathrm{s}, \sigma)}}$, so that we can factorize the full expression as $$D(k_\mathrm{s})=2k_\mathrm{s}^2\left[-\frac{2\gamma}{\gamma (\sigma+k_\mathrm{s})^2 +\sqrt{\gamma \Delta(k_\mathrm{s}, \sigma)}} + \sigma +3k_\mathrm{s}^2\right].$$ By definition of $\Delta(k_\mathrm{s}, \sigma)$ we have $\Delta(k_\mathrm{s}, \sigma)\geq \gamma (\sigma+k_\mathrm{s}^2)$  which implies $-\frac{2\gamma}{\gamma (\sigma+k_\mathrm{s})^2 +\sqrt{\gamma \Delta(k_\mathrm{s}, \sigma)}}\geq - (\sigma + k_\mathrm{s}^2)$, and defines a lower bound that is $\forall k_\mathrm{s}:\ D(k_\mathrm{s})\geq 2k_\mathrm{s}^2(2k_\mathrm
    {s}^2)=4k_\mathrm{s}^4\geq 0$. Thus, $D(k_\mathrm{s})$ is always positive and therefore the sign of $C(k_\mathrm{s}, k_\mathrm{s})$ depends solely on $(\sigma - k_\mathrm{s}^2)$, which switches sign to  positive when $|k_\mathrm{s}|< \sqrt{\sigma}$, which corresponds, by Lemma~\ref{lemma:symm_sqrt_d}, to the case in which the FNE is unique. By considering the opposite sign, we complete the claim.

\subsection{Proof of Theorem~\ref{thm:value_func_order}}
\label{appendix:value_func_order}
    Let us assume throughout the whole proof $a>0$ as the procedure is the same for both signs, and let us first show that the ordering $k_\mathrm{h_1}<k_\mathrm{s}<k_\mathrm{h_2}$ holds. By definition, $k_\mathrm{h_1}<k_\mathrm{h_2}<0$, both being roots of the quadratic $\mathcal{D}(k_i)$~\eqref{eq:quad_hyp}, while $k_\mathrm{s}$ is the second root among $k_1, k_\mathrm{s}, k_3$ of the cubic $\mathcal{C}(k)=0$~\eqref{eq:cubic}. Since $\mathcal{C}(k)$ is a cubic with positive leading coefficient and $k_{\mathrm{s}}$ is its middle root, it is sufficient to show $\mathcal{C}(k_\mathrm{h_1})>0$ and $\mathcal{C}(k_\mathrm{h_2})<0$. To do that, let us consider the polynomial $\mathcal{D}(k_i)$~\eqref{eq:quad_hyp} at the solutions $k_\mathrm{h_{1,2}}$ and rewrite it as $\gamma\sigma = \frac{k_\mathrm{h_{1,2}}(1-\gamma a(a+k_\mathrm{h_{1,2}}))}{a}$. Let us substitute the latter identity in $\mathcal{C}(k)=0$ to obtain $\mathcal{C}(k_\mathrm{h_{1,2}})=4\gamma k_\mathrm{h_{1,2}}\left(a+k_\mathrm{h_{1,2}}\right)\left(k_\mathrm{h_{1,2}}-\frac{1-\gamma a^2}{2\gamma a}\right)$, with $4\gamma k_\mathrm{h_{1,2}}\left(a+k_\mathrm{h_{1,2}}\right)<0\ \forall k_\mathrm{h_{1,2}}\in(-a, 0)$~\eqref{eq:sFNE_dom}. Moreover, by using the identity 
    \begin{equation}\label{eq:identity_hyp}
        \frac{k_\mathrm{h_1}+k_\mathrm{h_2}}{2}=\overline{k_\mathrm{h}}=\frac{1-\gamma a^2}{2\gamma a},
    \end{equation}
    denoting by $\overline{k_\mathrm{h}}$ the midpoint of the two roots $k_\mathrm{h_{1,2}}$, we obtain respectively for $k_\mathrm{h_1}$ and $k_\mathrm{h_2}$, that $\mathcal{C}(k_\mathrm{h_1})>0$ and $\mathcal{C}(k_\mathrm{h_2})<0$, obtaining the first claim. In addition to that, by evaluating $\mathcal{C}(k)$ at the midpoint $\overline{k_{\mathrm{h}}}$, we obtain $\mathcal{C}(\overline{k_\mathrm{h}})=\frac{(1-\gamma a^2)^2-4\gamma^2a^2\sigma}{4\gamma^2a^3}$ being strictly positive under condition~\eqref{eq:conditions_multiplicity}, which corresponds to the case of having three FNE. As a consequence, we have also proved that $k_\mathrm{h_1}<\overline{k_\mathrm{h}}<k_\mathrm{s}<k_\mathrm{h_2}$.
    
    Let us now focus on the solution's value functions~\eqref{eq:value_function} and observe that $a_\mathrm{cl}(k_\mathrm{h_1}, k_\mathrm{h_2})=a_\mathrm{cl}(k_\mathrm{h_2}, k_\mathrm{h_1})=a_\mathrm{cl}(k_h)=\frac{1}{\gamma a}>0$ thank to identity~\eqref{eq:identity_hyp}, and $a_\mathrm{cl}(k_\mathrm{s})=a+2 k_\mathrm{s}>\frac{1}{\gamma a}\iff k_\mathrm{s}> \overline{k_\mathrm{h}}$ thanks to the previous steps of the proof. As a consequence, the following inequality holds:
    \begin{equation}\label{eq:bound_closed_loop}
        \frac{1}{1-\gamma a_\mathrm{cl}^2(k_h)}\leq \frac{1}{1-\gamma a_\mathrm{cl}^2(k_\mathrm{s})}.
    \end{equation}
    Since $k_s<k_{\mathrm h_2}<0$, we have $k_{\mathrm h_2}^2<k_s^2$, and hence $\sigma+k_{\mathrm h_2}^2<\sigma+k_s^2$. As a result, we obtain the first inequality in the claim, i.e., $
        V_i(x_0; k_\mathrm{h_2}, k_\mathrm{h_1})\leq V_i(x_0; k_\mathrm{s}, k_\mathrm{s})$:
    \begin{equation}
        \underbrace{\frac{x_0^2 r(\sigma +k_\mathrm{h_2}^2)}{1-\gamma a_\mathrm{cl}^2(k_h)}}_{=V_i(x_0; k_\mathrm{h_2}, k_\mathrm{h_1})}\leq \frac{x_0^2 r(\sigma +k_\mathrm{s}^2)}{1-\gamma a_\mathrm{cl}^2(k_h)}\leq \underbrace{\frac{x_0^2 r(\sigma +k_\mathrm{s}^2)}{1-\gamma a_\mathrm{cl}^2(k_\mathrm{s})}}_{=V_i(x_0; k_\mathrm{s}, k_\mathrm{s})}. 
    \end{equation}
    We prove next the second inequality in the claim, i.e., $V_i(x_0; k_\mathrm{s}, k_\mathrm{s})\leq V_i(x_0;k_\mathrm{h_1}, k_\mathrm{h_2})$. Consider the cubic $\mathcal{C}(k)$~\eqref{eq:cubic}, solve it with respect to $\sigma$ to obtain $\sigma= \frac{k_\mathrm{s} \left(a^2 \gamma + 3\gamma a k_\mathrm{s} + 2 \gamma k_\mathrm{s}^2 - 1\right)}{\gamma \left(a + 2 k_\mathrm{s}\right)}.$ By substituting this in $\sigma+k_\mathrm{s}^2$, after a few simplifications we obtain $$\sigma+k_\mathrm{s}^2=k_\mathrm{s}\frac{\gamma(a+2k_\mathrm{s})^2-1}{\gamma (a+2k_\mathrm{s})}=k_\mathrm{s}\frac{\gamma a_\mathrm{cl}^2(k_\mathrm{s})-1}{\gamma a_\mathrm{cl}(k_\mathrm{s})}.$$ 
    As a consequence, the value function expression of the symmetric FNE is equal to 
    \begin{equation}
        V_i(x_0; k_\mathrm{s}, k_\mathrm{s})=-\frac{x_0^2 r k_\mathrm{s}}{\gamma a_\mathrm{cl}(k_\mathrm{s})}\geq0.
    \end{equation}
    Let us take the quadratic of the hyperbolic FNE \eqref{eq:quad_hyp}, solve it with respect to $\sigma$, which gives $\sigma=- a k_\mathrm{h_1} - k^2_\mathrm{h_1} + \frac{k_\mathrm{h_1}}{\gamma a}$, and substitute it in $\sigma+k^2_\mathrm{h_1}=- a k_\mathrm{h_1} + \frac{k_\mathrm{h_1}}{\gamma a}$. As a consequence, the value function of the hyperbolic FNE is equal to  
    \begin{equation}
        V_i(x_0; k_\mathrm{h_1}, k_\mathrm{h_2})=-x_0^2r k_\mathrm{h_1}a\geq0.
    \end{equation}
    To obtain the claim, one can recall the inequality $a_\mathrm{cl}(k_\mathrm{s})>\frac{1}{\gamma a}\iff \frac{1}{\gamma a_\mathrm{cl}(k_\mathrm{s})}<a$ and multiply both terms by $-x_0^2rk_\mathrm{s}>0$, to obtain the following
    \begin{equation}
        \underbrace{\frac{-x_0^2r k_\mathrm{s}}{\gamma a_\mathrm{cl}(k_\mathrm{s})}}_{=V_i(x_0; k_\mathrm{s}, k_\mathrm{s})}<-x_0^2r k_\mathrm{s}a<\underbrace{-x_0^2 rk_\mathrm{h_1}a}_{=V_i(x_0; k_\mathrm{h_1}, k_\mathrm{h_2})},
    \end{equation} 
    where we used the ordering $k_\mathrm{h_1}<k_\mathrm{s}\iff -k_\mathrm{h_1}>-k_\mathrm{s}$. The claim is obtained.

\bibliographystyle{IEEEtran}  
\bibliography{bib} 
\end{document}